\documentclass[12pt]{amsart}    % Specifies the document style.
\usepackage{graphicx, amsfonts, amsmath, amssymb, amsthm, pictexwd,color,ifsym} 
\usepackage[usenames,dvipsnames,svgnames,table]{}
\newtheorem{theorem}{Theorem}[section]

\newtheorem{lemma}[theorem]{Lemma}

\theoremstyle{definition}

\newtheorem{definition}{Definition}
\title{Presentation of the Motzkin Monoid}
\author{Kris Hatch, Megan Ly, Eliezer Posner}
\begin{document}           % End of preamble.
\maketitle                 % Produces the title.
%%%%%%%%%%%%%%%%%%%%%%%%%%%%%%%%%%%%%%%%%%%%%%%%%%%%%%%%%%%%%%%%
% For proofs, use \begin{proof} and \end{proof}
%%%%%%%%%%%%%%%%%%%%%%%%%%%%%%%%%%%%%%%%%%%%%%%%%%%%%%%%%%%%%%%%%
% The next paragraph is to produce the abstract.

%%%%%%%%%%%%%%%%%%%%%%%%%%%%%%%%%%%%%%%%%%%%%%%%%%%%%%%%%%%%%%%%
% Begin the paper. Sections will be numbered for you.
% Just enter a section title and then begin typing the section.
% You also can have subsections by using the command \subsection
%Beginning of Kris' section
\begin{abstract}
In 2010, Tom Halverson and Georgia Benkart introduced the Motzkin algebra, a generalization of the Temperley-Lieb algebra, whose elements are diagrams that can be multiplied by stacking one on top of the other. Halverson and Benkart gave a diagrammatic algorithm for decomposing any Motzkin diagram into diagrams of three subalgebras: the Right Planar Rook algebra, the Temperley-Lieb algebra, and the Left Planar Rook algebra. We first explore the Right and Left Planar Rook monoids, by finding presentations for these monoids by generators and relations, using a counting argument to prove that our relations suffice. We then turn to the newly-developed Motzkin monoid, where we describe Halverson's decomposition algorithm algebraically, find a presentation by generators and relations, and use a counting argument but with a much more sophisticated algorithm. 
\end{abstract}
\section{Introduction} %Rook and Planar Rook Monoids	
\hspace{6mm}	In this paper, the diagram monoids $R_n$ and one of its submonoids $P_n$ are first introduced in order to get an understanding of the diagram monoids $RP_n$ and $LP_n$,which are submonoids of $P_n$.  Presentations of the latter two are provided in this paper, as well as references to papers in which presentations for the former two are provided.  We then introduce the well known Temperley-Lieb Algebra, $TL_n$.   With these diagram monoids and the Temperley-Lieb algebra, 
we can give a presentation of the Motzkin Monoid, $M_n$.  In particular, we show that any diagram in $M_n$ can be decomposed into a product of the form $RTL$, where $R \in RP_n$, $T \in TL_n$, and $L \in LP_n$. (Note there is a particular form called ``standard" form mainly involving the placement of empty vertices.)   This is particularly important because of the inductive nature of our proof.\\  

	We start with a word in $RTL$ form, append a generator, $x_i$, to the end of the word, and then try to get it back into $RTL$ form.  We apply the relations of $M_n$ to first get $RTLx_i$ into the form $P_1TP_2$ (which we call $PTP$ form), where $P_1,P_2 \in P_n$, and $T \in TL_n$.  From this $PTP$ form, we then get to Minimal $RTL$ form, which is done by taking a diagram in $PTP$ form, and moving all the empty vertices to the right.  This is done using five lemmas (hop, burrow, slide, wallslide, and fuse wire), and also putting an order on $P_n$.  From here, we concern ourselves with the case that the edges in $T$ in minimal $RTL$ form have endpoints that are empty vertices of $R$ or $L$.  We show that we can put words of this form into standard $RTL$ form (where standard $RTL$ form is as defined in section 4.3.1).\\

	The preceding inductive proof tells us that the number of distinct words is equal to the number of standard words, and thus equal to the number of diagrams in $M_n$, giving us a presentation of $M_n$.

\vspace{10mm}

We define the \textit{Rook Monoid}, denoted $R_n$, as the set of one-to-one functions from a subset of $\{1, 2, . . . , n\}$ to a subset of $\{1, 2, . . . , n\}$.  These functions can be written as diagrams: a graph on two rows of $n$ vertices, each labeled from $1$ to $n$  from left to right.  On the graph we connect vertices on the top line to vertices on the bottom line.  Take for example the following diagram in $R_5$:\\
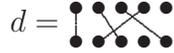
\begin{center}
$d={\beginpicture
\setcoordinatesystem units <0.3cm,0.15cm>         % sets scale 1st is how far apart dots are  east west 2nd is how far apart they are north south: I like anywhere from .25cm to .4cm for E-W, and .15-.2 for N S
\setplotarea x from 1 to 6, y from 0 to 3    % sets plot size up - 0 to 3 for y is fine, set x from 1 to n where n is number of top vertices
\linethickness=0.3pt                          % sets line thickness - I like 0.3 
\put{$\bullet$} at 1 -1 \put{$\bullet$} at 1 2 %puts bullet at  bottom, x coord, y - coord (If you need to add dots go putbullet at n -1 and n 2 respectively for a nice diagram size
\put{$\bullet$} at 2 -1 \put{$\bullet$} at 2 2
\put{$\bullet$} at 3 -1 \put{$\bullet$} at 3 2
\put{$\bullet$} at 4 -1 \put{$\bullet$} at 4 2
\put{$\bullet$} at 5 -1 \put{$\bullet$} at 5 2

\plot 1 -1 1 2 / %plots point from x coord to y coord of bullets
\plot 3 -1 2 2 /
\plot 5 -1 3 2 /
\plot 2 -1 4 2 /
%\plot 4 -1 5 2 /
\endpicture}$
\end{center}

The Rook Monoid is a \textit{monoid} under the operation of function composition, where the functions go from the top row of vertices to the bottom.  To perform function composition with two diagrams $d_1$, $d_2$, we place $d_1$ above $d_2$ and identify the vertices in the bottom row of $d_1$ with the corresponding vertices in the top row of $d_2$.  In terms of functions, if we consider $d_1$ as the function $g$, and $d_2$ as the function $f$, then the product $d_1d_2$\ is thought of as the composition $fg$, where if $x$ is in the range of $g$, and in the domain of $f$, we get that the pre-image of $x$, say $y$, in $g$ is in the domain of $fg$, and the image of $x$, say $z$, in $f$ is in the range of $fg$.  In particular, $y$ is in the domain of $fg$ if and only if there exists $x$ in the domain of $f$ such that $g(y)=x$.  This $y$ then corresponds to the value $z$ in the range of $fg$ such that $f^{-1}(z)=x$.  Note that in the diagram form we then draw an edge from $y$ on top to $z$ on bottom.

For example,\\ 

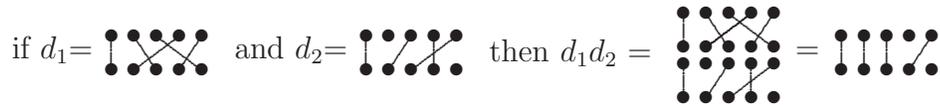
\begin{center}
if $d_1$=
{\beginpicture
\setcoordinatesystem units <0.3cm,0.15cm>         % sets scale 1st is how far apart dots are  east west 2nd is how far apart they are north south: I like anywhere from .25cm to .4cm for E-W, and .15-.2 for N S
\setplotarea x from 1 to 6, y from 0 to 3    % sets plot size up - 0 to 3 for y is fine, set x from 1 to n where n is number of top vertices
\linethickness=0.3pt                          % sets line thickness - I like 0.3 
\put{$\bullet$} at 1 -1 \put{$\bullet$} at 1 2 %puts bullet at  bottom, x coord, y - coord (If you need to add dots go putbullet at n -1 and n 2 respectively for a nice diagram size
\put{$\bullet$} at 2 -1 \put{$\bullet$} at 2 2
\put{$\bullet$} at 3 -1 \put{$\bullet$} at 3 2
\put{$\bullet$} at 4 -1 \put{$\bullet$} at 4 2
\put{$\bullet$} at 5 -1 \put{$\bullet$} at 5 2

\plot 1 -1 1 2 / %plots point from x coord to y coord of bullets
\plot 3 -1 2 2 /
\plot 5 -1 3 2 /
\plot 2 -1 4 2 /
\plot 4 -1 5 2 /
\endpicture}
and $d_2$=
{\beginpicture
\setcoordinatesystem units <0.3cm,0.15cm>         % sets scale 1st is how far apart dots are  east west 2nd is how far apart they are north south: I like anywhere from .25cm to .4cm for E-W, and .15-.2 for N S
\setplotarea x from 1 to 6, y from 0 to 3    % sets plot size up - 0 to 3 for y is fine, set x from 1 to n where n is number of top vertices
\linethickness=0.3pt                          % sets line thickness - I like 0.3 
\put{$\bullet$} at 1 -1 \put{$\bullet$} at 1 2 %puts bullet at  bottom, x coord, y - coord (If you need to add dots go putbullet at n -1 and n 2 respectively for a nice diagram size
\put{$\bullet$} at 2 -1 \put{$\bullet$} at 2 2
\put{$\bullet$} at 3 -1 \put{$\bullet$} at 3 2
\put{$\bullet$} at 4 -1 \put{$\bullet$} at 4 2
\put{$\bullet$} at 5 -1 \put{$\bullet$} at 5 2

\plot 1 -1 1 2 / %plots point from x coord to y coord of bullets
\plot 2 -1 3 2 /
\plot 3 -1 5 2 /
\plot 4 -1 4 2 /
\endpicture} 
{\beginpicture
\setcoordinatesystem units <0.3cm,0.15cm>         % sets scale 1st is how far apart dots are  east west 2nd is how far apart they are north south: I like anywhere from .25cm to .4cm for E-W, and .15-.2 for N S
\setplotarea x from 1 to 12, y from 0 to 3    % sets plot size up - 0 to 3 for y is fine, set x from 1 to n where n is number of top vertices
\linethickness=0.3pt                          % sets line thickness - I like 0.3
\put{$\bullet$} at 1 -3.5 \put{$\bullet$} at 1 -.5  \put{$\bullet$} at 1 1 \put{$\bullet$} at 1 4 %puts bullet at  bottom, x coord, y - coord (If you need to add dots go putbullet at n -1 and n 2 respectively for a nice diagram size
\put{$\bullet$} at 2 -3.5 \put{$\bullet$} at 2 -.5 \put{$\bullet$} at 2 1 \put{$\bullet$} at 2 4
\put{$\bullet$} at 3 -3.5 \put{$\bullet$} at 3 -.5 \put{$\bullet$} at 3 1 \put{$\bullet$} at 3 4
\put{$\bullet$} at 4 -3.5 \put{$\bullet$} at 4  -.5 \put{$\bullet$} at 4 1 \put{$\bullet$} at 4 4
\put{$\bullet$} at 5 -3.5 \put{$\bullet$} at 5 -.5 \put{$\bullet$} at 5 1 \put{$\bullet$} at 5 4
\put{then $d_1 d_2$ =} at -4 .5
\put{$\bullet$} at 8 2  \put{$\bullet$} at 8 -1
\put{$\bullet$} at 9 2  \put{$\bullet$} at 9 -1
\put{$\bullet$} at 10 2 \put{$\bullet$} at 10 -1
\put{$\bullet$} at 11 2  \put{$\bullet$} at 11 -1
\put{$\bullet$} at 12 2 \put{$\bullet$} at 12 -1
\put{$=$} at 6.5 .5

\plot 1 1 1 4 / %plots point from x coord to y coord of bullets
\plot 2 1 4 4 /
\plot 2 4 3 1 /
\plot 3 4 5 1 /
\plot 4 4 2 1 /
\plot 5 4 4 1 /

\plot 1 -.5 1 -3.5 /
\plot 2 -3.5 3 -.5 /
\plot 3 -3.5 5 -.5 /
\plot 4 -3.5 4 -.5 /

\plot 8 2 8 -1 /
\plot 9 2 9 -1 /
\plot 10 2 10 -1 /
\plot 12 2 11 -1 /
\endpicture}
\end{center}

Another way to think about the Rook Monoid, $R_n$, is as the set of $n \times n$ matrices which have entries in the set $\{0,1\}$, with the property that there is at most one $1$ in each row and each column.  Take for example, $R_2$, which consists of the matrices:

\[
 \begin{bmatrix}
  0 & 0 \\
  0 & 0 \\
 \end{bmatrix},
 \begin{bmatrix}
  1 & 0 \\
  0 & 0 \\
 \end{bmatrix},
  \begin{bmatrix}
  0 & 0 \\
  1 & 0 \\
 \end{bmatrix},
  \begin{bmatrix}
  0 & 1 \\
  0 & 0 \\
 \end{bmatrix},
  \begin{bmatrix}
  0 & 0 \\
  0 & 1 \\
 \end{bmatrix},
  \begin{bmatrix}
  1 & 0 \\
  0 & 1 \\
 \end{bmatrix}, \textrm{and}
  \begin{bmatrix}
  0 & 1 \\
  1 & 0 \\
 \end{bmatrix}
\]

Note that the elements in $R_n$ are in a one-to-one correspondence with the possible ways in which one can place non-attacking rooks on a $n \times n$ chess board.  We define the \textit{rank} of a diagram as the number of edges it has, or the number of $1$s  in the corresponding matrix.\\

If we desire a rook matrix of rank $k$, we choose $k$ columns and $k$ rows in a total of $\binom{n}{k}^2$ ways.  We then can choose to place the $1$s in the $k$ chosen rows/columns in a total of $k!$ ways (such that no row or column has two or more nonzero entries).  Summing over all the possible ranks up to and including $n$ gives rise to the order of $R_n$:
$$|R_n| = \sum\limits_{k=0}^n\binom{n}{k}^2 k!$$

The relationship between diagrams and matrices is given in a very natural way.  We connect the vertex in the $i^{\textrm{th}}$ position in the top row of a diagram to the $j^{\textrm{th}}$ position in the bottom row if and only if the corresponding matrix has a 1 in the ($i$,$j$)-position.  Take for example the following matrix-diagram correspondence in $R_5$:
$$
{\beginpicture
\setcoordinatesystem units <0.3cm,0.15cm>         % sets scale
\setplotarea x from 1 to 4, y from 0 to 3    % sets plot size up a larger x value will make smaller diagrams further apart.  need at minimum x from 1 to n where n = # top vertices
\linethickness=0.3pt                          % sets line thickness
\put{$\bullet$} at 1 -1 \put{$\bullet$} at 1 2 
\put{$\bullet$} at 2 -1 \put{$\bullet$} at 2 2
\put{$\bullet$} at 3 -1 \put{$\bullet$} at 3 2
\put{$\bullet$} at 4 -1 \put{$\bullet$} at 4 2
\put{$\bullet$} at 5 -1 \put{$\bullet$} at 5 2
\put{$\leftrightarrow$} at 7 0.5
\plot 1 -1 1 2  /
\plot 2 -1 2 2 /
\plot 3 -1 5 2 /
\plot 4 2 5 -1 /
\endpicture}
  \begin{bmatrix}
  1 & 0 & 0 & 0 & 0\\
  0 & 1 & 0 & 0 & 0\\
  0 & 0 & 0 & 0 & 0\\
  0 & 0 & 0 & 0 & 1\\
  0 & 0 & 1 & 0 & 0\\
 \end{bmatrix}
$$

Note that matrix multiplication is equivalent to diagram multiplication by stacking:
$$
{\beginpicture
\setcoordinatesystem units <0.3cm,0.15cm>         % sets scale 1st is how far apart dots are  east west 2nd is how far apart they are north south: I like anywhere from .25cm to .4cm for E-W, and .15-.2 for N S
\setplotarea x from 1 to 12, y from 0 to 3    % sets plot size up - 0 to 3 for y is fine, set x from 1 to n where n is number of top vertices
\linethickness=0.3pt                          % sets line thickness - I like 0.3
\put{$\bullet$} at 1 -3.5 \put{$\bullet$} at 1 -.5  \put{$\bullet$} at 1 1 \put{$\bullet$} at 1 4 %puts bullet at  bottom, x coord, y - coord (If you need to add dots go putbullet at n -1 and n 2 respectively for a nice diagram size
\put{$\bullet$} at 2 -3.5 \put{$\bullet$} at 2 -.5 \put{$\bullet$} at 2 1 \put{$\bullet$} at 2 4
\put{$\bullet$} at 3 -3.5 \put{$\bullet$} at 3 -.5 \put{$\bullet$} at 3 1 \put{$\bullet$} at 3 4
\put{$\bullet$} at 4 -3.5 \put{$\bullet$} at 4  -.5 \put{$\bullet$} at 4 1 \put{$\bullet$} at 4 4
\put{$\bullet$} at 5 -3.5 \put{$\bullet$} at 5 -.5 \put{$\bullet$} at 5 1 \put{$\bullet$} at 5 4
\put{$d_1=$} at -1 2.5
\put{$d_2=$} at -1 -1.5
\put{$=d_1d_2$} at 15 .5
\put{$\bullet$} at 8 2  \put{$\bullet$} at 8 -1
\put{$\bullet$} at 9 2  \put{$\bullet$} at 9 -1
\put{$\bullet$} at 10 2 \put{$\bullet$} at 10 -1
\put{$\bullet$} at 11 2  \put{$\bullet$} at 11 -1
\put{$\bullet$} at 12 2 \put{$\bullet$} at 12 -1
\put{$=$} at 6.5 .5

\plot 1 1 1 4 / %plots point from x coord to y coord of bullets
\plot 2 1 4 4 /
\plot 2 4 3 1 /
\plot 3 4 5 1 /
\plot 4 4 2 1 /
\plot 5 4 4 1 /

\plot 1 -.5 1 -3.5 /
\plot 2 -3.5 3 -.5 /
\plot 3 -3.5 5 -.5 /
\plot 4 -3.5 4 -.5 /

\plot 8 2 8 -1 /
\plot 9 2 9 -1 /
\plot 10 2 10 -1 /
\plot 12 2 11 -1 /
\endpicture}
$$
\begin{center}
is equivalent to
\end{center}
$$
  \begin{bmatrix}
  1 & 0 & 0 & 0 & 0\\
  0 & 0 & 1 & 0 & 0\\
  0 & 0 & 0 & 0 & 1\\
  0 & 1 & 0 & 0 & 0\\
  0 & 0 & 0 & 1 & 0\\
 \end{bmatrix}
   \begin{bmatrix}
  1 & 0 & 0 & 0 & 0\\
  0 & 0 & 0 & 0 & 0\\
  0 & 1 & 0 & 0 & 0\\
  0 & 0 & 0 & 1 & 0\\
  0 & 0 & 1 & 0 & 0\\
 \end{bmatrix}
 =
   \begin{bmatrix}
  1 & 0 & 0 & 0 & 0\\
  0 & 1 & 0 & 0 & 0\\
  0 & 0 & 1 & 0 & 0\\
  0 & 0 & 0 & 0 & 0\\
  0 & 0 & 0 & 1 & 0\\
 \end{bmatrix}
$$
\vspace{5mm}

A presentation of the rook monoid is provided on page 339 in \cite{rook}.
\vspace{10mm}

The \textit{Planar Rook Monoid}, $P_n$, is the set of order-preserving one-to-one functions from a subset of $\{1, 2, \dots , n\}$ to a subset of $\{1, 2, \dots , n\}$.   These order-preserving functions correspond to those diagrams that can be drawn with edges that do not cross.  For example, the set $P_2$ consists of the following diagrams:
\indent
\begin{center}
$P_2$\ =\ 
${\beginpicture
\setcoordinatesystem units <0.3cm,0.15cm>         % sets scale
\setplotarea x from 1 to 4, y from 0 to 3    % sets plot size up a larger x value will make smaller diagrams further apart.  need at minimum x from 1 to n where n = # top vertices
\linethickness=0.3pt                          % sets line thickness
\put{$\bullet$} at 1 -1 \put{$\bullet$} at 1 2
\put{$\bullet$} at 2  -1 \put{$\bullet$} at 2 2
\plot 1 -1 1 2  /
\plot 2 -1 2 2 /
\endpicture} $
${\beginpicture
\setcoordinatesystem units <0.3cm,0.15cm>         % sets scale
\setplotarea x from 1 to 4, y from 0 to 3    % sets plot size up a larger x value will make smaller diagrams further apart.  need at minimum x from 1 to n where n = # top vertices
\linethickness=0.3pt                          % sets line thickness
\put{$\bullet$} at 1 -1 \put{$\bullet$} at 1 2
\put{$\bullet$} at 2  -1 \put{$\bullet$} at 2 2
\plot 1 -1 1 2  /
\endpicture} $
${\beginpicture
\setcoordinatesystem units <0.3cm,0.15cm>         % sets scale
\setplotarea x from 1 to 4, y from 0 to 3    % sets plot size up a larger x value will make smaller diagrams further apart.  need at minimum x from 1 to n where n = # top vertices
\linethickness=0.3pt                          % sets line thickness
\put{$\bullet$} at 1 -1 \put{$\bullet$} at 1 2
\put{$\bullet$} at 2  -1 \put{$\bullet$} at 2 2
\plot 2 -1 2 2 /
\endpicture} $
${\beginpicture
\setcoordinatesystem units <0.3cm,0.15cm>         % sets scale
\setplotarea x from 1 to 4, y from 0 to 3    % sets plot size up a larger x value will make smaller diagrams further apart.  need at minimum x from 1 to n where n = # top vertices
\linethickness=0.3pt                          % sets line thickness
\put{$\bullet$} at 1 -1 \put{$\bullet$} at 1 2
\put{$\bullet$} at 2  -1 \put{$\bullet$} at 2 2
\plot 1 -1 2 2  /
\endpicture} $
${\beginpicture
\setcoordinatesystem units <0.3cm,0.15cm>         % sets scale
\setplotarea x from 1 to 4, y from 0 to 3    % sets plot size up a larger x value will make smaller diagrams further apart.  need at minimum x from 1 to n where n = # top vertices
\linethickness=0.3pt                          % sets line thickness
\put{$\bullet$} at 1 -1 \put{$\bullet$} at 1 2
\put{$\bullet$} at 2  -1 \put{$\bullet$} at 2 2
\plot 2 -1 1 2  /
\endpicture} $
${\beginpicture
\setcoordinatesystem units <0.3cm,0.15cm>         % sets scale
\setplotarea x from 1 to 4, y from 0 to 3    % sets plot size up a larger x value will make smaller diagrams further apart.  need at minimum x from 1 to n where n = # top vertices
\linethickness=0.3pt                          % sets line thickness
\put{$\bullet$} at 1 -1 \put{$\bullet$} at 1 2
\put{$\bullet$} at 2  -1 \put{$\bullet$} at 2 2
\endpicture}$
\end{center}

For a diagram $d$, we define 
%the subsets of $\mathbb{N}$, 
$\tau(d)$ and $\beta(d)$ to be the sets containing the indices of the vertices of $d$ which are incident to an edge on top and on bottom respectively.  For example,

\begin{center} if $d$\ =\ 
${\beginpicture
\setcoordinatesystem units <0.3cm,0.15cm>         % sets scale 1st is how far apart dots are  east west 2nd is how far apart they are north south: I like anywhere from .25cm to .4cm for E-W, and .15-.2 for N S
\setplotarea x from 1 to 6, y from 0 to 3    % sets plot size up - 0 to 3 for y is fine, set x from 1 to n where n is number of top vertices
\linethickness=0.3pt                          % sets line thickness - I like 0.3 
\put{$\bullet$} at 1 -1 \put{$\bullet$} at 1 2 %puts bullet at  bottom, x coord, y - coord (If you need to add dots go putbullet at n -1 and n 2 respectively for a nice diagram size
\put{$\bullet$} at 2 -1 \put{$\bullet$} at 2 2
\put{$\bullet$} at 3 -1 \put{$\bullet$} at 3 2
\put{$\bullet$} at 4 -1 \put{$\bullet$} at 4 2
\put{$\bullet$} at 5 -1 \put{$\bullet$} at 5 2

\plot 1 -1 1 2 / %plots point from x coord to y coord of bullets
\plot 3 -1 2 2 /
\plot 5 -1 3 2 /
\plot 4 -1 5 2 /
\endpicture} $,
then $\tau$($d$) = \{1,2,3,5\} and $\beta$($d$) = \{$1'$,$3'$,$4'$,$5'$\} \\
\end{center}

\noindent where we label the top vertices from 1 to $n$ and the bottom vertices from $1'$ to $n'$.  For a planar rook diagram $d$, there is only one way to connect the vertices by edges, thus these sets $\tau$($d$) and $\beta$($d$) uniquely determine $d$.\\

Notice that the product of two planar rook diagrams is planar (seen easily through diagram multiplication), thus $P_n$ is a submonoid of $R_n$.   To obtain the order of $P_n$, as in $R_n$, we choose $k$ columns, and $k$ rows in $\binom{n}{k}^2$, however as stated above this choosing determines a unique diagram, thus the total number of planar rook diagrams is:
$$|P_n| = \sum\limits_{k=0}^n\binom{n}{k}^2 $$
 %%%%%%%%%GENERATORS AND RELATIONS%%%%%%%%%%%%%
\subsection{Generators and Relations of the Planar Rook Monoid}
\vspace{5mm}

 Let $l_i$ be the element of $P_n$ such that $\tau(l_i)=[n]\setminus\{i+1\}$ and $\beta(l_i)=[n]\setminus\{i\}=\{1,2,\ldots,i-1,i+i,\ldots,n\}$, as shown below:

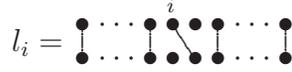
\begin{center}
${\beginpicture
\setcoordinatesystem units <0.3cm,0.15cm>
\setplotarea x from 1 to 3.8, y from 0 to 3
\linethickness=0.3pt
%\put{$\bullet$} at 1 -1 \put{$\bullet$} at 1 2
\put{$\bullet$} at 2 -1 \put{$\bullet$} at 2 2
%\put{$\bullet$} at 4 -1 \put{$\bullet$} at 4 2
\put{$\bullet$} at 5 -1 \put{$\bullet$} at 5 2
\put{$\bullet$} at 6 -1 \put{$\bullet$} at 6 2
\put{$\bullet$} at 7 -1 \put{$\bullet$} at 7 2
\put{$\bullet$} at 8 -1 \put{$\bullet$} at 8 2

\put{$\bullet$} at 11 -1 \put{$\bullet$} at 11 2
\put{$\cdots$} at 3.625 2 \put{$\cdots$} at 3.625 -1
\put{$\cdots$} at 9.625 2
\put{$\cdots$} at 9.625 -1

%\plot 1 2 1 -1 /
\plot 2 2 2 -1 /
\plot 5 2 5 -1 /
\plot 8 2 8 -1 /
\plot 11 -1 11 2 /
\plot 6 2 7 -1 /

\put{$\scriptstyle{i\,}$} at 6 3.75
\put{$l_i =$} at 0 .5
\endpicture}$
\end{center}

\noindent Let $r_i$ be the element of $P_n$ such that $\tau(r_i)=[n]\setminus\{i\}$ and $\beta(r_i)=[n]\setminus\{i+1\}$ as shown below:
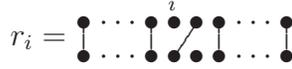
\begin{center}
${\beginpicture
\setcoordinatesystem units <0.3cm,0.15cm>
\setplotarea x from 1 to 3.8, y from 0 to 3
\linethickness=0.3pt
%\put{$\bullet$} at 1 -1 \put{$\bullet$} at 1 2
\put{$\bullet$} at 2 -1 \put{$\bullet$} at 2 2
%\put{$\bullet$} at 4 -1 \put{$\bullet$} at 4 2
\put{$\bullet$} at 5 -1 \put{$\bullet$} at 5 2
\put{$\bullet$} at 6 -1 \put{$\bullet$} at 6 2
\put{$\bullet$} at 7 -1 \put{$\bullet$} at 7 2
\put{$\bullet$} at 8 -1 \put{$\bullet$} at 8 2

\put{$\bullet$} at 11 -1 \put{$\bullet$} at 11 2
\put{$\cdots$} at 3.625 2 \put{$\cdots$} at 3.625 -1
\put{$\cdots$} at 9.625 2
\put{$\cdots$} at 9.625 -1

%\plot 1 2 1 -1 /
\plot 2 2 2 -1 /
\plot 5 2 5 -1 /
\plot 8 2 8 -1 /
\plot 6 -1 7 2 /
\plot 11 2 11 -1 /

\put{$\scriptstyle{i\,}$} at 6 3.75
\put{$r_i =$} at 0 .5
\endpicture}$
\end{center}

\noindent As proven in \cite{Kathy}, every planar rook diagram can be written as a product of $l_i$'s and $r_i$'s.

The following relations hold for all $i$ such that the terms in the relation are defined:

\begin{enumerate}
\item $l_i^3=l_i^2=r_i^2=r_i^3$
\item a) $l_il_{i+1}l_i=l_il_{i+1}=l_{i+1}l_il_{i+1}$\\
b) $r_ir_{i+1}r_i=r_{i+1}r_i=r_{i+1}r_ir_{i+1}$
\item a) $r_il_ir_i=r_i$\\
b) $l_ir_il_i=l_i$
\item a) $l_{i+1}r_il_i=l_{i+1}r_i$\\
b) $r_{i-1}l_ir_i=r_{i-1}l_i$\\
c) $r_il_ir_{i+1}=l_ir_{i+1}$\\
d) $l_ir_il_{i-1}=r_il_{i-1}$
\item $l_ir_i=l_{i+1}r_{i+1}$
\item If $|i-j|\geq 2$, then $r_il_j=l_jr_i, r_ir_j=r_jr_i, l_il_j=l_jl_i$
\end{enumerate}
These relations can easily be verified by drawing the diagram products they refer to. Furthermore, as proven in \cite{Kathy}, these relations suffice to completely characterize $P_n$.

\section{Right Planar Rook Monoid}
We define the \textit{Right Planar Rook Monoid} to be the submonoid of the Planar Rook Monoid which contains all the diagrams with the property that the top vertex of each edge is directly above or above and to the right of the bottom vertex.  Similarly, the \textit{Left Planar Rook Monoid} is the submonoid where the top vertex of each edge is above or above and to the left of the bottom vertex.  We denote these monoids $RP_n$ and $LP_n$ respectively.  For example, 

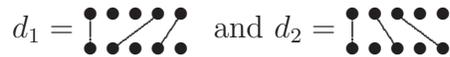
\begin{center}
$d_1={\beginpicture
\setcoordinatesystem units <0.3cm,0.15cm>         % sets scale 1st is how far apart dots are  east west 2nd is how far apart they are north south: I like anywhere from .25cm to .4cm for E-W, and .15-.2 for N S
\setplotarea x from 1 to 6, y from 0 to 3    % sets plot size up - 0 to 3 for y is fine, set x from 1 to n where n is number of top vertices
\linethickness=0.3pt                          % sets line thickness - I like 0.3 
\put{$\bullet$} at 1 -1 \put{$\bullet$} at 1 2 %puts bullet at  bottom, x coord, y - coord (If you need to add dots go putbullet at n -1 and n 2 respectively for a nice diagram size
\put{$\bullet$} at 2 -1 \put{$\bullet$} at 2 2
\put{$\bullet$} at 3 -1 \put{$\bullet$} at 3 2
\put{$\bullet$} at 4 -1 \put{$\bullet$} at 4 2
\put{$\bullet$} at 5 -1 \put{$\bullet$} at 5 2

\plot 1 -1 1 2 / %plots point from x coord to y coord of bullets
%\plot 3 -1 2 2 /
%\plot 5 -1 3 2 /
\plot 2 -1 4 2 /
\plot 4 -1 5 2 /
\endpicture}$
and $d_2={\beginpicture
\setcoordinatesystem units <0.3cm,0.15cm>         % sets scale 1st is how far apart dots are  east west 2nd is how far apart they are north south: I like anywhere from .25cm to .4cm for E-W, and .15-.2 for N S
\setplotarea x from 1 to 6, y from 0 to 3    % sets plot size up - 0 to 3 for y is fine, set x from 1 to n where n is number of top vertices
\linethickness=0.3pt                          % sets line thickness - I like 0.3 
\put{$\bullet$} at 1 -1 \put{$\bullet$} at 1 2 %puts bullet at  bottom, x coord, y - coord (If you need to add dots go putbullet at n -1 and n 2 respectively for a nice diagram size
\put{$\bullet$} at 2 -1 \put{$\bullet$} at 2 2
\put{$\bullet$} at 3 -1 \put{$\bullet$} at 3 2
\put{$\bullet$} at 4 -1 \put{$\bullet$} at 4 2
\put{$\bullet$} at 5 -1 \put{$\bullet$} at 5 2

\plot 1 -1 1 2 / %plots point from x coord to y coord of bullets
\plot 3 -1 2 2 /
\plot 5 -1 3 2 /
%\plot 2 -1 4 2 /
%\plot 4 -1 5 2 /
\endpicture}$
\end{center}

\noindent are in $RP_5$ and $LP_5$ respectively.\\

 \noindent We will now derive some facts about $RP_n$.

\subsection{Cardinality}
First, we prove a closed form for the size of $RP_n$. Through numerical experimentation we find that for the first few values of $n$, we get:
\[
|RP_1| = 2, |RP_2| = 5, |RP_3| = 14, |RP_4| = 42, |RP_5|=132.
\]
$|RP_1|$ has diagrams: \\

$
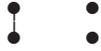
{\beginpicture
\setcoordinatesystem units <0.2cm,0.2cm>         % sets scale
\setplotarea x from 1 to 5, y from 0 to 3    % sets plot size up
\linethickness=0.3pt                          % sets line thickness
\put{$\bullet$} at 1 -1 \put{$\bullet$} at 1 1
\plot 1 -1 1 1  /
\endpicture} $ 
${\beginpicture
\setcoordinatesystem units <0.2cm,0.2cm>         % sets scale
\setplotarea x from 1 to 5, y from 0 to 3    % sets plot size up
\linethickness=0.3pt                          % sets line thickness
\put{$\bullet$} at 1 -1 \put{$\bullet$} at 1 1
\endpicture} $ \\

\noindent $|RP_2|$ has diagrams:\\

$
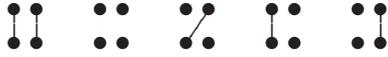
{\beginpicture
\setcoordinatesystem units <0.3cm,0.15cm>         % sets scale
\setplotarea x from 1 to 4, y from 0 to 3    % sets plot size up a larger x value will make smaller diagrams further apart.  need at minimum x from 1 to n where n = # top vertices
\linethickness=0.3pt                          % sets line thickness
\put{$\bullet$} at 1 -1 \put{$\bullet$} at 1 2
\put{$\bullet$} at 2  -1 \put{$\bullet$} at 2 2
\plot 1 -1 1 2  /
\plot 2 -1 2 2 /
\endpicture} $ 
${\beginpicture
\setcoordinatesystem units <0.3cm,0.15cm>         % sets scale 1st is how far apart dots are  east west 2nd is how far apart they are north south: I like anywhere from .25cm to .4cm for E-W, and .15-.2 for N S
\setplotarea x from 1 to 4, y from 0 to 3    % sets plot size up - 0 to 3 for y is fine, set x from 1 to n where n is number of top vertices
\linethickness=0.3pt                          % sets line thickness - I like 0.3 
\put{$\bullet$} at 1 -1 \put{$\bullet$} at 1 2 %puts bullet at  bottom, x coord, y - coord (If you need to add dots go putbullet at n -1 and n 2 respectively for a nice diagram size
\put{$\bullet$} at 2 -1 \put{$\bullet$} at 2 2
\endpicture} $
${\beginpicture
\setcoordinatesystem units <0.3cm,0.15cm>         % sets scale 1st is how far apart dots are  east west 2nd is how far apart they are north south: I like anywhere from .25cm to .4cm for E-W, and .15-.2 for N S
\setplotarea x from 1 to 4, y from 0 to 3    % sets plot size up - 0 to 3 for y is fine, set x from 1 to n where n is number of top vertices
\linethickness=0.3pt                          % sets line thickness - I like 0.3 
\put{$\bullet$} at 1 -1 \put{$\bullet$} at 1 2 %puts bullet at  bottom, x coord, y - coord (If you need to add dots go putbullet at n -1 and n 2 respectively for a nice diagram size
\put{$\bullet$} at 2 -1 \put{$\bullet$} at 2 2
\plot 1 -1 2 2 / %plots point from x coord to y coord of bullets.
\endpicture} $
${\beginpicture
\setcoordinatesystem units <0.3cm,0.15cm>         % sets scale 1st is how far apart dots are  east west 2nd is how far apart they are north south: I like anywhere from .25cm to .4cm for E-W, and .15-.2 for N S
\setplotarea x from 1 to 4, y from 0 to 3    % sets plot size up - 0 to 3 for y is fine, set x from 1 to n where n is number of top vertices
\linethickness=0.3pt                          % sets line thickness - I like 0.3 
\put{$\bullet$} at 1 -1 \put{$\bullet$} at 1 2 %puts bullet at  bottom, x coord, y - coord (If you need to add dots go putbullet at n -1 and n 2 respectively for a nice diagram size
\put{$\bullet$} at 2 -1 \put{$\bullet$} at 2 2
\plot 1 -1 1 2 / %plots point from x coord to y coord of bullets
\endpicture} $
${\beginpicture
\setcoordinatesystem units <0.3cm,0.15cm>         % sets scale 1st is how far apart dots are  east west 2nd is how far apart they are north south: I like anywhere from .25cm to .4cm for E-W, and .15-.2 for N S
\setplotarea x from 1 to 4, y from 0 to 3    % sets plot size up - 0 to 3 for y is fine, set x from 1 to n where n is number of top vertices
\linethickness=0.3pt                          % sets line thickness - I like 0.3 
\put{$\bullet$} at 1 -1 \put{$\bullet$} at 1 2 %puts bullet at  bottom, x coord, y - coord (If you need to add dots go putbullet at n -1 and n 2 respectively for a nice diagram size
\put{$\bullet$} at 2 -1 \put{$\bullet$} at 2 2
\plot 2 -1 2 2 / %plots point from x coord to y coord of bullets
\endpicture} $\\

\noindent $|RP_3|$ has diagrams:\\

$
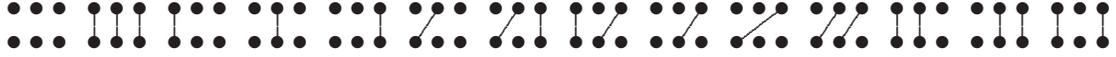
{\beginpicture
\setcoordinatesystem units <0.3cm,0.15cm>         % sets scale 1st is how far apart dots are  east west 2nd is how far apart they are north south: I like anywhere from .25cm to .4cm for E-W, and .15-.2 for N S
\setplotarea x from 1 to 3.8, y from 0 to 3    % sets plot size up - 0 to 3 for y is fine, set x from 1 to n where n is number of top vertices
\linethickness=0.3pt                          % sets line thickness - I like 0.3 
\put{$\bullet$} at 1 -1 \put{$\bullet$} at 1 2 %puts bullet at  bottom, x coord, y - coord (If you need to add dots go putbullet at n -1 and n 2 respectively for a nice diagram size
\put{$\bullet$} at 2 -1 \put{$\bullet$} at 2 2
\put{$\bullet$} at 3 -1 \put{$\bullet$} at 3 2
\endpicture} $
${\beginpicture
\setcoordinatesystem units <0.3cm,0.15cm>         % sets scale 1st is how far apart dots are  east west 2nd is how far apart they are north south: I like anywhere from .25cm to .4cm for E-W, and .15-.2 for N S
\setplotarea x from 1 to 3.8, y from 0 to 3    % sets plot size up - 0 to 3 for y is fine, set x from 1 to n where n is number of top vertices
\linethickness=0.3pt                          % sets line thickness - I like 0.3 
\put{$\bullet$} at 1 -1 \put{$\bullet$} at 1 2 %puts bullet at  bottom, x coord, y - coord (If you need to add dots go putbullet at n -1 and n 2 respectively for a nice diagram size
\put{$\bullet$} at 2 -1 \put{$\bullet$} at 2 2
\put{$\bullet$} at 3 -1 \put{$\bullet$} at 3 2
\plot 1 -1 1 2 / %plots point from x coord to y coord of bullets
\plot 2 -1 2 2 / %plots point from x coord to y coord of bullets
\plot 3 -1 3 2 / %plots point from x coord to y coord of bullets
\endpicture} $
${\beginpicture
\setcoordinatesystem units <0.3cm,0.15cm>         % sets scale 1st is how far apart dots are  east west 2nd is how far apart they are north south: I like anywhere from .25cm to .4cm for E-W, and .15-.2 for N S
\setplotarea x from 1 to 3.8, y from 0 to 3    % sets plot size up - 0 to 3 for y is fine, set x from 1 to n where n is number of top vertices
\linethickness=0.3pt                          % sets line thickness - I like 0.3 
\put{$\bullet$} at 1 -1 \put{$\bullet$} at 1 2 %puts bullet at  bottom, x coord, y - coord (If you need to add dots go putbullet at n -1 and n 2 respectively for a nice diagram size
\put{$\bullet$} at 2 -1 \put{$\bullet$} at 2 2
\put{$\bullet$} at 3 -1 \put{$\bullet$} at 3 2
\plot 1 -1 1 2 / %plots point from x coord to y coord of bullets
\endpicture} $
${\beginpicture
\setcoordinatesystem units <0.3cm,0.15cm>         % sets scale 1st is how far apart dots are  east west 2nd is how far apart they are north south: I like anywhere from .25cm to .4cm for E-W, and .15-.2 for N S
\setplotarea x from 1 to 3.8, y from 0 to 3    % sets plot size up - 0 to 3 for y is fine, set x from 1 to n where n is number of top vertices
\linethickness=0.3pt                          % sets line thickness - I like 0.3 
\put{$\bullet$} at 1 -1 \put{$\bullet$} at 1 2 %puts bullet at  bottom, x coord, y - coord (If you need to add dots go putbullet at n -1 and n 2 respectively for a nice diagram size
\put{$\bullet$} at 2 -1 \put{$\bullet$} at 2 2
\put{$\bullet$} at 3 -1 \put{$\bullet$} at 3 2
\plot 2 -1 2 2 / %plots point from x coord to y coord of bullets
\endpicture} $
${\beginpicture
\setcoordinatesystem units <0.3cm,0.15cm>         % sets scale 1st is how far apart dots are  east west 2nd is how far apart they are north south: I like anywhere from .25cm to .4cm for E-W, and .15-.2 for N S
\setplotarea x from 1 to 3.8, y from 0 to 3    % sets plot size up - 0 to 3 for y is fine, set x from 1 to n where n is number of top vertices
\linethickness=0.3pt                          % sets line thickness - I like 0.3 
\put{$\bullet$} at 1 -1 \put{$\bullet$} at 1 2 %puts bullet at  bottom, x coord, y - coord (If you need to add dots go putbullet at n -1 and n 2 respectively for a nice diagram size
\put{$\bullet$} at 2 -1 \put{$\bullet$} at 2 2
\put{$\bullet$} at 3 -1 \put{$\bullet$} at 3 2
\plot 3 -1 3 2 / %plots point from x coord to y coord of bullets
\endpicture} $
${\beginpicture
\setcoordinatesystem units <0.3cm,0.15cm>         % sets scale 1st is how far apart dots are  east west 2nd is how far apart they are north south: I like anywhere from .25cm to .4cm for E-W, and .15-.2 for N S
\setplotarea x from 1 to 3.8, y from 0 to 3    % sets plot size up - 0 to 3 for y is fine, set x from 1 to n where n is number of top vertices
\linethickness=0.3pt                          % sets line thickness - I like 0.3 
\put{$\bullet$} at 1 -1 \put{$\bullet$} at 1 2 %puts bullet at  bottom, x coord, y - coord (If you need to add dots go putbullet at n -1 and n 2 respectively for a nice diagram size
\put{$\bullet$} at 2 -1 \put{$\bullet$} at 2 2
\put{$\bullet$} at 3 -1 \put{$\bullet$} at 3 2
\plot 1 -1 2 2 / %plots point from x coord to y coord of bullets
\endpicture} $
${\beginpicture
\setcoordinatesystem units <0.3cm,0.15cm>         % sets scale 1st is how far apart dots are  east west 2nd is how far apart they are north south: I like anywhere from .25cm to .4cm for E-W, and .15-.2 for N S
\setplotarea x from 1 to 3.8, y from 0 to 3    % sets plot size up - 0 to 3 for y is fine, set x from 1 to n where n is number of top vertices
\linethickness=0.3pt                          % sets line thickness - I like 0.3 
\put{$\bullet$} at 1 -1 \put{$\bullet$} at 1 2 %puts bullet at  bottom, x coord, y - coord (If you need to add dots go putbullet at n -1 and n 2 respectively for a nice diagram size
\put{$\bullet$} at 2 -1 \put{$\bullet$} at 2 2
\put{$\bullet$} at 3 -1 \put{$\bullet$} at 3 2
\plot 1 -1 2 2 / %plots point from x coord to y coord of bullets
\plot 3 -1 3 2 / %plots point from x coord to y coord of bullets
\endpicture} $
${\beginpicture
\setcoordinatesystem units <0.3cm,0.15cm>         % sets scale 1st is how far apart dots are  east west 2nd is how far apart they are north south: I like anywhere from .25cm to .4cm for E-W, and .15-.2 for N S
\setplotarea x from 1 to 3.8, y from 0 to 3    % sets plot size up - 0 to 3 for y is fine, set x from 1 to n where n is number of top vertices
\linethickness=0.3pt                          % sets line thickness - I like 0.3 
\put{$\bullet$} at 1 -1 \put{$\bullet$} at 1 2 %puts bullet at  bottom, x coord, y - coord (If you need to add dots go putbullet at n -1 and n 2 respectively for a nice diagram size
\put{$\bullet$} at 2 -1 \put{$\bullet$} at 2 2
\put{$\bullet$} at 3 -1 \put{$\bullet$} at 3 2
\plot 1 -1 1 2 / %plots point from x coord to y coord of bullets
\plot 2 -1 3 2 /
\endpicture} $
${\beginpicture
\setcoordinatesystem units <0.3cm,0.15cm>         % sets scale 1st is how far apart dots are  east west 2nd is how far apart they are north south: I like anywhere from .25cm to .4cm for E-W, and .15-.2 for N S
\setplotarea x from 1 to 3.8, y from 0 to 3    % sets plot size up - 0 to 3 for y is fine, set x from 1 to n where n is number of top vertices
\linethickness=0.3pt                          % sets line thickness - I like 0.3 
\put{$\bullet$} at 1 -1 \put{$\bullet$} at 1 2 %puts bullet at  bottom, x coord, y - coord (If you need to add dots go putbullet at n -1 and n 2 respectively for a nice diagram size
\put{$\bullet$} at 2 -1 \put{$\bullet$} at 2 2
\put{$\bullet$} at 3 -1 \put{$\bullet$} at 3 2
\plot 2 -1 3 2 / %plots point from x coord to y coord of bullets
\endpicture} $
${\beginpicture
\setcoordinatesystem units <0.3cm,0.15cm>         % sets scale 1st is how far apart dots are  east west 2nd is how far apart they are north south: I like anywhere from .25cm to .4cm for E-W, and .15-.2 for N S
\setplotarea x from 1 to 3.8, y from 0 to 3    % sets plot size up - 0 to 3 for y is fine, set x from 1 to n where n is number of top vertices
\linethickness=0.3pt                          % sets line thickness - I like 0.3 
\put{$\bullet$} at 1 -1 \put{$\bullet$} at 1 2 %puts bullet at  bottom, x coord, y - coord (If you need to add dots go putbullet at n -1 and n 2 respectively for a nice diagram size
\put{$\bullet$} at 2 -1 \put{$\bullet$} at 2 2
\put{$\bullet$} at 3 -1 \put{$\bullet$} at 3 2
\plot 1 -1 3 2 / %plots point from x coord to y coord of bullets
\endpicture} $
${\beginpicture
\setcoordinatesystem units <0.3cm,0.15cm>         % sets scale 1st is how far apart dots are  east west 2nd is how far apart they are north south: I like anywhere from .25cm to .4cm for E-W, and .15-.2 for N S
\setplotarea x from 1 to 3.8, y from 0 to 3    % sets plot size up - 0 to 3 for y is fine, set x from 1 to n where n is number of top vertices
\linethickness=0.3pt                          % sets line thickness - I like 0.3 
\put{$\bullet$} at 1 -1 \put{$\bullet$} at 1 2 %puts bullet at  bottom, x coord, y - coord (If you need to add dots go putbullet at n -1 and n 2 respectively for a nice diagram size
\put{$\bullet$} at 2 -1 \put{$\bullet$} at 2 2
\put{$\bullet$} at 3 -1 \put{$\bullet$} at 3 2
\plot 1 -1 2 2 / %plots point from x coord to y coord of bullets
\plot 2 -1 3 2 /
\endpicture} $
${\beginpicture
\setcoordinatesystem units <0.3cm,0.15cm>         % sets scale 1st is how far apart dots are  east west 2nd is how far apart they are north south: I like anywhere from .25cm to .4cm for E-W, and .15-.2 for N S
\setplotarea x from 1 to 3.8, y from 0 to 3    % sets plot size up - 0 to 3 for y is fine, set x from 1 to n where n is number of top vertices
\linethickness=0.3pt                          % sets line thickness - I like 0.3 
\put{$\bullet$} at 1 -1 \put{$\bullet$} at 1 2 %puts bullet at  bottom, x coord, y - coord (If you need to add dots go putbullet at n -1 and n 2 respectively for a nice diagram size
\put{$\bullet$} at 2 -1 \put{$\bullet$} at 2 2
\put{$\bullet$} at 3 -1 \put{$\bullet$} at 3 2
\plot 2 -1 2 2 / %plots point from x coord to y coord of bullets
\plot 1 -1 1 2 /
\endpicture} $
${\beginpicture
\setcoordinatesystem units <0.3cm,0.15cm>         % sets scale 1st is how far apart dots are  east west 2nd is how far apart they are north south: I like anywhere from .25cm to .4cm for E-W, and .15-.2 for N S
\setplotarea x from 1 to 3.8, y from 0 to 3    % sets plot size up - 0 to 3 for y is fine, set x from 1 to n where n is number of top vertices
\linethickness=0.3pt                          % sets line thickness - I like 0.3 
\put{$\bullet$} at 1 -1 \put{$\bullet$} at 1 2 %puts bullet at  bottom, x coord, y - coord (If you need to add dots go putbullet at n -1 and n 2 respectively for a nice diagram size
\put{$\bullet$} at 2 -1 \put{$\bullet$} at 2 2
\put{$\bullet$} at 3 -1 \put{$\bullet$} at 3 2
\plot 2 -1 2 2 / %plots point from x coord to y coord of bullets
\plot 3 -1 3 2 /
\endpicture} $
${\beginpicture
\setcoordinatesystem units <0.3cm,0.15cm>         % sets scale 1st is how far apart dots are  east west 2nd is how far apart they are north south: I like anywhere from .25cm to .4cm for E-W, and .15-.2 for N S
\setplotarea x from 1 to 3.8, y from 0 to 3    % sets plot size up - 0 to 3 for y is fine, set x from 1 to n where n is number of top vertices
\linethickness=0.3pt                          % sets line thickness - I like 0.3 
\put{$\bullet$} at 1 -1 \put{$\bullet$} at 1 2 %puts bullet at  bottom, x coord, y - coord (If you need to add dots go putbullet at n -1 and n 2 respectively for a nice diagram size
\put{$\bullet$} at 2 -1 \put{$\bullet$} at 2 2
\put{$\bullet$} at 3 -1 \put{$\bullet$} at 3 2
\plot 1 -1 1 2 / %plots point from x coord to y coord of bullets
\plot 3 -1 3 2 /
\endpicture} $ \\

\noindent {We notice that these cardinalities are Catalan numbers, and indeed we can prove that this pattern continues:
\begin{theorem}
$|RP_n| = C_{n+1} = \frac{\binom{2n+2}{n+1}}{n+2}$ 
\end{theorem}
\begin{proof}
We begin by defining a new way to encode any diagram $d \in RP_n$. To do this we create ordered pairs of the form $(i,j)$ for each edge in $d$ connecting the $i$th vertex on top to the $j$th vertex on bottom. We then gather all such pairs into a set along with the pair $(n+1,n+1)$. We impose a well ordering on this set given by $(a,b) \leq (c,d) \iff a \leq c$. For example, the diagram given by:
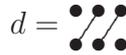
\begin{center}
$d=
{\beginpicture
\setcoordinatesystem units <0.3cm,0.15cm>         % sets scale 1st is how far apart dots are  east west 2nd is how far apart they are north south: I like anywhere from .25cm to .4cm for E-W, and .15-.2 for N S
\setplotarea x from 1 to 3.8, y from 0 to 3    % sets plot size up - 0 to 3 for y is fine, set x from 1 to n where n is number of top vertices
\linethickness=0.3pt                          % sets line thickness - I like 0.3 
\put{$\bullet$} at 1 -1 \put{$\bullet$} at 1 2 %puts bullet at  bottom, x coord, y - coord (If you need to add dots go putbullet at n -1 and n 2 respectively for a nice diagram size
\put{$\bullet$} at 2 -1 \put{$\bullet$} at 2 2
\put{$\bullet$} at 3 -1 \put{$\bullet$} at 3 2
\plot 1 -1 2 2 / %plots point from x coord to y coord of bullets
\plot 2 -1 3 2 /
\endpicture}$
\end{center}
yields the poset $\{ (2,1) \leq (3,2) \leq (4,4)\}$. 
Now we begin by exhibiting a bijection from $RP_n$ to the set of all $2(n+1)$-sequences of $\pm 1$s, with exactly $n+1$ of each, such that no partial sum is negative. The set of such sequences is known to  have cardinality $C_{n+1}$.\\

Define a function from $RP_n$ to the set is as follows: given any diagram $d = \{(a_1,b_1) \leq \ldots \leq (a_k,b_k)\}$, start our sequence with $a_1$ $1$s followed by $b_1$ $(-1)$s. Inductively, append to this sequence $a_{i} - a_{i-1}$ $1$s and $b_{i} - b_{i-1}$ $(-1)$s for $2 \leq i \leq k$. To see that the resulting sequence is indeed in our set, note that diagrams in $RP_n$ must consist of ordered pairs with their first component greater than or equal to their second. This, in combination with the fact that the first coordinate of every pair is appended first as $1$s, makes it impossible to have a negative partial sum.\\

To see that this function is a bijection, note that there is a natural inverse. Given any sequence, we create the first ordered pair by counting off all of the $1$s until the first $(-1)$, as well as the number of $(-1)$s up to the next $1$. These numbers will be the first and second coordinates of the first ordered pair, $(a_1,b_1)$. Inductively we count the $i$th set of $1$s and the $i$th set of $(-1)$s. These numbers would then be added to $a_{i-1}$ and $b_{i-1}$ to give an ordered pair $(a_i, b_i)$.  Therefore the cardinality of $RP_n$ is equal to $C_{n+1}$.
\end{proof}

\subsection{Generators}
Let $r_i$ be defined as in Subsection \ref{copying kathy}. Let $p_i$ be the element of $RP_n$ such that $\tau(p_i)=\beta(p_i)=[n]\setminus\{i\}$ as shown:
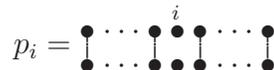
\begin{center}
${\beginpicture
\setcoordinatesystem units <0.3cm,0.15cm>
\setplotarea x from 1 to 3.8, y from 0 to 3
\linethickness=0.3pt
%\put{$\bullet$} at 1 -1 \put{$\bullet$} at 1 2
\put{$\bullet$} at 2 -1 \put{$\bullet$} at 2 2
%\put{$\bullet$} at 4 -1 \put{$\bullet$} at 4 2
\put{$\bullet$} at 5 -1 \put{$\bullet$} at 5 2
\put{$\bullet$} at 6 -1 \put{$\bullet$} at 6 2
\put{$\bullet$} at 7 -1 \put{$\bullet$} at 7 2
\put{$\bullet$} at 10 -1 \put{$\bullet$} at 10 2
\put{$\cdots$} at 3.625 2 \put{$\cdots$} at 3.625 -1
\put{$\cdots$} at 8.625 2
\put{$\cdots$} at 8.625 -1

%\plot 1 2 1 -1 /
\plot 2 2 2 -1 /
\plot 5 2 5 -1 /
\plot 7 2 7 -1 /
\plot 10 2 10 -1 /

\put{$\scriptstyle{i\,}$} at 6 3.75
\put{$p_i =$} at 0 .5
\endpicture}$
\end{center}
This yields the following theorem.

\begin{theorem}\label{generators}Every diagram of $RP_n$ can be written as a product of $r_i\ (1 \leq i < n)$ and $p_i\ (1 \leq i \leq n)$.
\end{theorem}

\begin{proof}Suppose $d$ is a diagram in $RP_n$ with rank $k$.  Let $\tau(d)= \{a_1, \ldots, a_k\}$ with $a_1 < \ldots < a_k$ and $\beta(d)= \{b_1, \ldots, b_k\}$ with $b_1 < \ldots < b_k$. 
%Let $\tau(d)_i$ and $\beta(d)_i$ be the $i$th elements of $\tau(d)$ and $\beta(d)$ respectively. For example, if $d$=
%${\beginpicture
%\setcoordinatesystem units <0.3cm,0.15cm>         % sets scale
%\setplotarea x from 1 to 4, y from 0 to 3    % sets plot size up a larger x value will make smaller diagrams further apart.  need at minimum x from 1 to n where n = # top vertices
%\linethickness=0.3pt                          % sets line thickness
%\put{$\bullet$} at 1 -1 \put{$\bullet$} at 1 2
%\put{$\bullet$} at 2  -1 \put{$\bullet$} at 2 2
%\put{$\bullet$} at 3 -1 \put{$\bullet$} at 3 2
%\put{$\bullet$} at 4  -1 \put{$\bullet$} at 4 2
%\put{$\bullet$} at 5 -1 \put{$\bullet$} at 5 2
%\plot 1 -1 2 2 /
%\plot 3 -1 3 2 /
%\plot 4 -1 5 2 /
%\endpicture} $ 
%then $\tau(d)_1=2$, $\tau(d)_2=3$, $\tau(d)_3=5$, $\beta(d)_1=1$, $\beta(d)_2=3$, $\beta(d)_3=4$.  
Define $R_{a}^{a}$ to be the identity diagram.
For $1\leq b<a \leq n$, define $R_{b}^{a}$ to be $$R_{b}^a= r_{a-1}\cdot r_{a-2}\dots r_{b+1}\cdot r_b.$$ Then $R_{b}^a$ is the diagram consisting of an edge from vertex $b'$ to $a$ and vertical edges connecting all vertices to the right of $a$ or to the left of $b'$.
Let, $$d'=\prod_{i=1}^{k} R_{b_i}^{a_i} \cdot \prod_{i\notin\beta(d)} p_i\, .$$
All that remains to be shown is that $d'=d$.  Consider a node $b_i' \in \beta(d)$.  All diagrams in $\prod_{i\notin\beta(d)} p_i$ have a vertical edge at $b_i'$.  So do the diagrams $R_{b_j}^{a_j}$ for $j=k, k-1, \ldots, i+1$.  Then $R_{b_i}^{a_i}$ connects nodes $b_i'$ and $a_i$.  Finally, $R_{b_j}^{a_j}$ has a vertical edge at $a_i$ for $j= i-1, i-2 \ldots 1$.  Thus $d'$ connects nodes $b_i'$ and $a_i$.  If $j' \notin \beta(d)$ then the product of $p_i$'s ensures $j'\notin \beta(d')$.  Thus $d=d'$ is a word in our generators.
%Notice that any two vertices, $a, b'$, which are connected in $d$ are connected in the product by definition of $R_{b'}^a$.  Moreover, if $i'\not\in\beta(d)$, then $i'$ is not in the bottom set of the product since we multiply by $p_i$ and any factors after $p_i$ fix the $i$th vertex.  Similarly, if $i\not\in\tau(d)$, then $i$ is not in the top set of the product. To see this, consider two cases:
%\begin{itemize}
%\item[] Case 1: $i\not\in\beta(d)$.
%\\Multiplying by $p_i$ causes the $i$ and $i'$ vertices to become empty (not incident to an edge).  Furthermore,  the only way we could make the $i$th vertex nonempty is by multiplying by $R^{a,i}$ for any $a$, however, we never multiply by $R^{a,i}$.  Thus the $i$th vertex remains empty.

%\item[] Case 2: $i\in\beta(d)$ but $i\not\in\tau(d)$.
%\\Let $j$ be the element of $\tau(d)$ which is connected to $i'$ in $d$, then $R^{j,i'}$ is a factor in the product so after multiplying by $R^{i,j}$ the $i$ vertex is empty. Additionally, the product contains no factor of $R^{i,k}$, for any $k$, thus the $i$ vertex remains empty, giving us that $i\not\in\tau(p)$.
%\end{itemize}
%Thus any diagram $d$ can be written as a product of $r_i$ and $p_i$.
\end{proof}
Note $p_i$ is not considered a generator in $P_n$, since $p_i=r_il_i$, for all $1\leq i < n$ and that $p_i=l_{i-1}r_{i-1}$ for all $1<i\leq n$.

%%%%%%%%%%%%%%%%%%%%%%%%%%%%%%
\subsection{Relations}

\begin{theorem} \label{RPn relations} The Right Planar Rook Monoid is generated by $r_i$ and $p_i$ subject to the following relations:
\begin{enumerate}
\item $r_i r_{i+1} r_i = r_{i+1} r_i r_{i+1}$
\item $r_i^2 = r_i^3 = r_i p_i = p_i p_{i+1}$
\item $r_i = p_i r_i = r_i p_{i+1}$
\item $p_i r_{i+1} = r_{i+1} p_i$
\item $p_i^2 = p_i$
\item $p_i p_j = p_j p_i$
\item If $|i-j| \geq 2$ then $p_i r_j = r_j p_i, and\  r_i r_j = r_j r_i$
\end{enumerate}
\end{theorem}

\begin{proof}
Theorem \ref{generators} shows that $r_i$ and $p_i$ generate all of $RP_n$, so every diagram can be written as a formal word.  We must show that there are at most as many distinct equivalence classes of words as there are diagrams.  To this end we define the standard word as follows:

A word $W$ is said to be in standard form if there exist subsets of $[n]$, $S$ and $T$ with the following three properties:\\
\indent $\bullet$ The sets $S$ and $T$ have equal cardinality, say $k$.\\
Let $S_i$ and $T_i$ be the $i$th elements $S$ and $T$ respectively. For example, if $S=\{1,4,7\}$, then $S_1=1, S_2=4, S_3=7$.\\
\indent $\bullet$ For all $1\leq i \leq k, T_i\leq S_i$.\\
Let $R^a_b$ be defined as is Theorem \ref{generators}.\\
\indent $\bullet$ $W=RP$, where $R=\prod_{i=1}^k R^{S_i}_{T_i}$ and $P=\prod_{i\notin T} p_i$

Since every diagram corresponds to exactly one standard word, we must show that each formal word is equal to a standard word.  Given a standard word $W=R P$ and a generator $x_i \in \{r_i, p_i\}$ we show that we can standardize the product $Wx_i$.  We will assume $W$ is not the empty word, since if $W$ is the empty word, after appending either of the generators the product is obviously equivalent to a standard word.  First consider the product $Wp_i$.  We have two cases:
\begin{itemize}
\item[] Case 1: $i \notin \beta(d)$.  Then necessarily $p_i$ appears in the product $W$.  Since $p_i p_j = p_j p_i$, $W$ can be written in the form: $W=R P_0 p_i$ where $P_0$ does not contain $p_i$
.  Then using the relation $p_j^2= p_j$ we have that $Wp_i=  R P_0 p_i p_i = R P_0 p_i=W$.
\item[] Case 2: $i \in \beta(d)$.  This means that $p_i \notin P$.  So if $r_i \notin R$ then $Wp_i = RPp_i = RP'$.  If $r_i \in R$ then there exists $j$ such that $r_i \in R_{i}^{j}$ for $i\leq j \leq k$.  We commute the $p_i$ next to the $r_i$ and use the relation $r_s p_s = p_s p_{s+1}$ to remove the $r$'s in $R_i^j$.  Then we commute the $p$'s past the $r$'s.  Thus,
%\textcolor{red}
{
\begin{eqnarray*}
Wp_i&=&RPp_i\\
&=& R_{b_1}^{a_1} \dots R_{i}^{j} \dots R_{b_k}^{a_k}P p_i\\
&=& R_{b_1}^{a_1} \dots r_{j} \dots r_ip_i \dots R_{b_k}^{a_k} P\ (p_s p_t=p_t p_s\ \mathrm{and}\ p_sr_t=r_t p_s\ \textrm{for}\ |t-s| \geq 2)\\
&=& R_{b_1}^{a_1} \dots p_i \dots p_{j+1} \dots R_{b_k}^{a_k} P\ (r_s p_s = p_s p_{s+1})\\
&=& R_{b_1}^{a_1} \dots R_{b_k}^{a_k} Pp_i \dots p_{j+1}\ (p_sr_t=r_t p_s)\\
&=&R'P'.
\end{eqnarray*}
}%end red
\end{itemize}

Now consider $Wr_i$.  Let $p_j$ be the rightmost letter of $W.$ There are four cases to consider:
\begin{itemize}
\item[] Case 1: $|i-j|\geq 2$. Since any of the generators commute when the indices are at least two apart, $r_i$ can be moved to the left.

\item[] Case 2:  $i=j-1$. Since $p_{j} \cdot r_{j-1}= p_{j-1} \cdot p_j$, we can eliminate $r_i$.

\item[] Case 3: $i=j+1$. Since $p_{j} \cdot r_{j+1} = r_{j+1} p_j$ we can move $r_i$ to the left.

\item[] Case 4: $i=j$. Using the relation $p_j \cdot r_j = p_j \cdot p_{j+1}$, we can eliminate $r_j$.
\end{itemize}
These relations can then be applied repeatedly until the $r$'s have been eliminated or commuted past $P$.  If we eliminate the $r$'s then we have that $Wr_i= RPr_i= RP'$.  If the $r$'s are commuted past $P$ then we must show that words of the form $Rr_lP$ can be standardized. Let $r_j$ be the letter immediately to the left of $r_l$. We again have four cases depending on the difference between $l$ and $j$:
\begin{itemize}
\item[] Case 1: $|l-j|\geq 2$. Since any of the generators commute when the indices are at least two apart, $r_l$ can be moved to the left.

\item[] Case 2: $l=j-1$. We have that $Rr_lP= R_{b_1}^{a_1}\cdots R^{a_m}_jr_{j-1}P= R_{b_1}^{a_1}\cdots  R_{j-1}^{a_m}P$, which is already in standard form.

\item[] Case 3: $l=j+1$. This implies that $R r_l P=  R_{b_1}^{a_1}\cdots  R_{j}^{a_m}r_{j+1}P$, which is already in standard form.

\item[] Case 4: $l=j$.  Note that $r_l\cdot r_j=r_l^2=p_l\cdot p_{l+1}$. Moving the $p_l\cdot p_{l+1}$ to the right, through $r$'s, is equivalent to moving $r$'s to the left through $p$'s, which is doable, as proven above.
\end{itemize}

This process can be repeated as long as the word is not in standard form. So we have proven that $Wx_i$ can be standardized. By induction, this implies that when a string of arbitrary length is appended to a word in standard form the product can be standardized. Thus any word can be put into standard form, implying that our relations are sufficient.
\end{proof}
\subsection{Presentation of $LP_n$}
%We define the function $*:M_n\rightarrow M_n$ by $d \mapsto d^*$ where 

For every diagram $d$ in $RP_n$, let $d^*$ be the diagram obtained by interchanging the vertices in the top and bottom rows of $d$ while maintaining edge connections. Notice that $$l_i^*=r_i,\ r_i^*=l_i,\ p_i^*=p_i$$ for all $i$ such that the terms are defined. Thus the function $*$ is an antiisomorphism and an involution, meaning that $*$ is an isomorphism from $RP_n$ to the opposite of $RP_n$ (which is $LP_n$), and $*$ is its own inverse. Therefore $LP_n$ is antiisomorphic to $RP_n$, and from every theorem about $RP_n$ one can easily derive a corresponding theorem regarding $LP_n$. In particular, it can easily be shown that $LP_n$ is generated by $l_i$ and $p_i$ and that it is completely characterized by relations analogous to those in theorem \ref{RPn relations}.

%%%%%%%%%%%%%%%%%%%
%%%%TEMPERLEY -LIEB ALGEBRA$%%%%%%%
%%%%%%%%%%%%%%%%%%%

\section{Temperley-Lieb Algebra}
A \textit{Temperley-Lieb diagram} has two rows of $n$ vertices, connected by non-crossing edges.  Unlike planar rook diagrams, Temperley-Lieb diagrams do not have any empty vertices and they may have \textit{horizontal edges}.  A horizontal edge is an edge connecting two distinct vertices in the same row. 
The collection of linear combinations of Temperley-Lieb diagrams forms the Temperley Lieb algebra, denoted $TL_n(x)$.

To multiply two Temperley-Lieb diagrams $d_1$ and $d_2$, we place $d_1$ on top of $d_2$, and identify the bottom vertices of $d_1$ with the top vertices of $d_2$.  Additionally, horizontal edges between vertices in the bottom set of $d_1$ and top set of $d_2$ that form a closed loop produce a \textit{bubble}.  We eliminate the bubbles and multiply the product diagram by a factor of $x$ for each bubble that is removed.

As discussed in \cite{temperley-lieb}, every Temperley-Lieb diagram can be written as the product of diagrams $t_i$, where $${\beginpicture
\setcoordinatesystem units <0.3cm,0.15cm>
\setplotarea x from 1 to 3.8, y from 0 to 3
\linethickness=0.3pt
%\put{$\bullet$} at 1 -1 \put{$\bullet$} at 1 2
\put{$\bullet$} at 2 -1 \put{$\bullet$} at 2 2
%\put{$\bullet$} at 4 -1 \put{$\bullet$} at 4 2
\put{$\bullet$} at 5 -1 \put{$\bullet$} at 5 2
\put{$\bullet$} at 6 -1 \put{$\bullet$} at 6 2
\put{$\bullet$} at 7 -1 \put{$\bullet$} at 7 2
\put{$\bullet$} at 8 -1 \put{$\bullet$} at 8 2

\put{$\bullet$} at 11 -1 \put{$\bullet$} at 11 2
\put{$\cdots$} at 3.625 2 \put{$\cdots$} at 3.625 -1
\put{$\cdots$} at 9.625 2
\put{$\cdots$} at 9.625 -1

%\plot 1 2 1 -1 /
\plot 2 2 2 -1 /
\plot 5 2 5 -1 /
\plot 8 2 8 -1 /
\plot 11 2 11 -1 /
\setquadratic
\plot 6 2 6.5 1 7 2 /
\plot 6 -1 6.5 0 7 -1 /

\put{$\scriptstyle{i\,}$} at 6 3.75
\put{$t_i =$} at 0 .5
\endpicture} $$ 

for $1\leq i <n$.  The algebra $TL_n(x)$ is completely characterized by the following relations
\begin{eqnarray*}
t_it_j&=&t_jt_i \textrm{ if $|i-j|\geq 2$}\\
t_i&=&t_it_jt_i \textrm{ if $|i-j|=1$}\\
t_i^2&=&xt_i\textrm{ for all $1\leq i< n$}.
\end{eqnarray*}

These relations are sufficient to show that any string of Temperley-Lieb generators is equivalent to a string of the form $$x^m(t_{i_1}t_{i_1-1}\dots t_{j_1})(t_{i_2}t_{i_2-1}\dots t_{j_2})\dots(t_{i_p}t_{i_p-1}\dots t_{j_p}),$$ for some integers $m, p, i_1,j_1,\dots,i_p, j_p$ with the properties that
\begin {eqnarray*}
m,p&\geq& 0\\
1&\leq& i_1 < i_2 < \dots < i_p < n\\ 
1&\leq& j_1 < j_2 < \dots < j_p < n\\
j_1&\leq& i_1, j_2 \leq i_2, \dots, j_p \leq i_p.
\end{eqnarray*}
The above form is called the \textit{standard form} of a Temperley-Lieb word. \\

We will need to use some non-standard forms of words in the Temperley-Lieb algebra.
\begin{lemma}
For any Temperley-Lieb diagram, $d$, in which vertex $i'$ is connected to $j$ or $j'$, we can express $d$ as a word in one of the following five forms:

\begin{enumerate}\label{okay}
\item If $i'$ is connected to $j'$ with $i<j$. Then $d$ can be written in the form\\
$T (t_i t_{i+2}\cdots t_{j-1}) (t_{i+1} t_{i+3} \cdots t_{j-2}) T'$, where $T \in \langle t_1 , \ldots, t_{n-1} \rangle, T'\in \langle t_{i+1}, \ldots, t_{j-2} \rangle$. Note that if $j'=i'+1$ then the word is of the form $T t_i$.

\item If $i'$ is connected to $j'$ with $j<i$. Then $d$ can be written in the form\\
$T (t_j t_{j+2}\cdots t_{i-1}) (t_{j+1} t_{j+3} \cdots t_{i-2}) T'$, where $T \in \langle t_1 , \ldots, t_{n-1} \rangle, T'\in \langle t_{j+1}, \ldots, t_{i-2} \rangle$. Note that if $i'=j'+1$ then the word is of the form $T t_j$.

\item If $i'$ is connected to $j$ with $j<i$. Then $d$ can be written in the form \\
$T (t_{j+1} t_{j+3} \cdots t_{i-1})(t_j t_{j+2} \cdots t_{i-2}) T'$, where $T \in \langle t_{j+1}, \ldots, t_{n-1} \rangle, T' \in \langle t_1, \ldots, t_{i-2} \rangle.$

\item If $i'$ is connected to $j$ with $j>i$.  Then $d$ can be written in the form \\
$T (t_i t_{i+2} \cdots t_{j-2})(t_{i+1} t_{i+3} \cdots t_{j-1}) T'$ where $T \in \langle t_1, \ldots, t_{j-2} \rangle, T' \in \langle t_{i+1}, \ldots, t_{n-1} \rangle $.

\item If $i'$ is connected to $j$ with $j=i$.  Then $d$ can be written in the form $T T'$, where  $T \in \langle t_{1}, t_{2}, \ldots, t_{i-2} \rangle$,  $T' \in \langle t_{i+1}, t_{i+2}, \ldots, t_{n-1} \rangle$.
\end{enumerate}
\end{lemma}
\begin{proof}
Given that these are the only five ways the vertex $i'$ is connected to $j$ or $j'$, by drawing diagrams, one can see that these are indeed the ways in which Temperley-Lieb diagrams can be drawn.  For example, in \textit{1.}, simply take the diagram $t_i t_{i+2} \cdots t_{j-1}$ and perform multiplication by stacking it above the diagram $t_{i+1} t_{i+3} \cdots t_{j-2}$, along with any chosen diagram $T \in \langle t_1 , \ldots, t_{n-1} \rangle$ above this product, and  $T'\in \langle t_{i+1}, \ldots, t_{j-2} \rangle$ below, and one can verify that indeed this gives the desired connection from $i'$ to $j'$ with $i < j$.
\end{proof}

In this paper we will consider multiplication of diagrams, but we will not consider linear combinations of diagrams. For simplicity, we will therefore take $x$ to be $1$ and refer exclusively to the monoid $TL_n$, although all results can be generalized to the algebra $TL_n(x)$, by multiplying by $x$ where appropriate.

%%%%%%%%%%%%%%%%%%%%%%%%%%%%%%%%%%%%%%%%%%%%%%%
\section{Motzkin Monoid}

%As seen in \cite{motzkin}, for each $k=0,1,2,\ldots$, the Motzkin numbers, $M^{(k)}$, are defined by the generating function:
%\begin{center}
%$M(t)=\sum\limits_{k\geq0}M^{(k)}t^n= \frac{1-t - \sqrt{1-2t-3t^2}}{2t^2}$\\
%\end{center}
%Satisfying 
%\begin{center}
%$M(t)=1+tM(t)+t^2M^2(t)$
%\end{center}

%\noindent The Motzkin numbers appear in several places, including:\\
%\indent $\bullet$\ The number of ways of drawing any number of nonintersecting chords among $n$ points on a circle.\\
%\indent $\bullet$\ The number of walks on $\{0,1,\ldots\}$ with n steps from $\{-1,0-1\}$ starting and ending at 0.\\

%\noindent An extensive list can be found at: http://oeis.org/A001006\\
%\vspace{10mm}

A Motzkin diagram is similar to a Temperley-Lieb diagram, however Motzkin diagrams may have \textit{empty} vertices (vertices not incident to an edge).  For example:\\

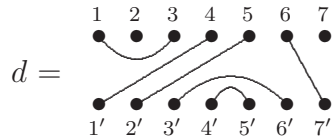
\begin{center}
$d$\ ={\beginpicture
\setcoordinatesystem units <0.5cm,0.3cm>
\setplotarea x from 0 to 8, y from -1 to 2
\put{$\bullet$} at  1 -1  \put{$\bullet$} at  1 2
\put{$\bullet$} at  2 -1  \put{$\bullet$} at  2 2
\put{$\bullet$} at  3 -1  \put{$\bullet$} at  3 2
\put{$\bullet$} at  4 -1  \put{$\bullet$} at  4 2
\put{$\bullet$} at  5 -1  \put{$\bullet$} at  5 2
\put{$\bullet$} at  6 -1  \put{$\bullet$} at  6 2
\put{$\bullet$} at  7 -1  \put{$\bullet$} at  7 2
\plot  4 2  1 -1 /
\plot 5 2 2 -1 /
\plot 6 2 7 -1  /
\setquadratic
\plot 1 2 2 1 3 2 /
\plot 3 -1 4.5 .25 6 -1 /
\plot 4 -1 4.5 -.25 5 -1 /
\put{$\scriptstyle{1\,}$} at 1 3.0
\put{$\scriptstyle{2\,}$} at 2 3.0
\put{$\scriptstyle{3\,}$} at 3 3.0
\put{$\scriptstyle{4\,}$} at 4 3.0
\put{$\scriptstyle{5\,}$} at 5 3.0
\put{$\scriptstyle{6\,}$} at 6 3.0
\put{$\scriptstyle{7\,}$} at 7 3.0
\put{$\scriptstyle{1'\,}$} at 1 -2
\put{$\scriptstyle{2'\,}$} at 2 -2
\put{$\scriptstyle{3'\,}$} at 3 -2
\put{$\scriptstyle{4'\,}$} at 4 -2
\put{$\scriptstyle{5'\,}$} at 5 -2
\put{$\scriptstyle{6'\,}$} at 6 -2
\put{$\scriptstyle{7'\,}$} at 7 -2
\endpicture}\\
\end{center}

We define the \textit{Motzkin Monoid}, $M_n$, as the set of Motzkin diagrams with $n$ vertices.
For example $M_2$ consists of the following diagrams:\\

\indent $M_2\ = $
%%%%% M2 DIAGRAMS%%%%%%%%
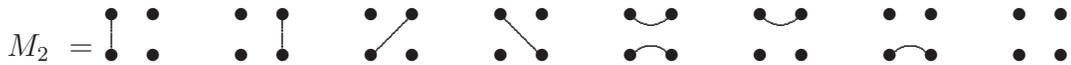
{\beginpicture
\setcoordinatesystem units <0.55cm,0.55cm>
\setplotarea x from 1 to 2, y from 0 to 1
\put{$\bullet$} at  1 0  \put{$\bullet$} at  1 1
\put{$\bullet$} at  2 0  \put{$\bullet$} at  2 1
\plot 1 0 1 1 /
\endpicture}  \qquad
{\beginpicture
\setcoordinatesystem units <0.55cm,0.55cm>
\setplotarea x from 1 to 2, y from 0 to 1
\put{$\bullet$} at  1 0  \put{$\bullet$} at  1 1
\put{$\bullet$} at  2 0  \put{$\bullet$} at  2 1
\plot 2 0 2 1 /
\endpicture} \qquad
{\beginpicture
\setcoordinatesystem units <0.55cm,0.55cm>
\setplotarea x from 1 to 2, y from 0 to 1
\put{$\bullet$} at  1 0  \put{$\bullet$} at  1 1
\put{$\bullet$} at  2 0  \put{$\bullet$} at  2 1
\plot 1 0 2 1 /
\endpicture} \qquad
{\beginpicture
\setcoordinatesystem units <0.55cm,0.55cm>
\setplotarea x from 1 to 2, y from 0 to 1
\put{$\bullet$} at  1 0  \put{$\bullet$} at  1 1
\put{$\bullet$} at  2 0  \put{$\bullet$} at  2 1
\plot 1 1 2 0 /
\endpicture} \qquad
{\beginpicture
\setcoordinatesystem units <0.55cm,0.55cm>
\setplotarea x from 1 to 2, y from 0 to 1
\put{$\bullet$} at  1 0  \put{$\bullet$} at  1 1
\put{$\bullet$} at  2 0  \put{$\bullet$} at  2 1
\setquadratic
\plot 1 1 1.5 .75 2 1 /
\plot 1 0 1.5 .25 2 0 /
\endpicture} \qquad
{\beginpicture
\setcoordinatesystem units <0.55cm,0.55cm>
\setplotarea x from 1 to 2, y from 0 to 1
\put{$\bullet$} at  1 0  \put{$\bullet$} at  1 1
\put{$\bullet$} at  2 0  \put{$\bullet$} at  2 1
\setquadratic
\plot 1 1 1.5 .75 2 1 /
\endpicture} \qquad
{\beginpicture
\setcoordinatesystem units <0.55cm,0.55cm>
\setplotarea x from 1 to 2, y from 0 to 1
\put{$\bullet$} at  1 0  \put{$\bullet$} at  1 1
\put{$\bullet$} at  2 0  \put{$\bullet$} at  2 1
\setquadratic
\plot 1 0 1.5 .25 2 0 /
\endpicture} \qquad
{\beginpicture
\setcoordinatesystem units <0.55cm,0.55cm>
\setplotarea x from 1 to 2, y from 0 to 1
\put{$\bullet$} at  1 0  \put{$\bullet$} at  1 1
\put{$\bullet$} at  2 0  \put{$\bullet$} at  2 1
\endpicture}
\vspace{5mm}

To multiply two Motzkin diagrams, $d_1$ and $d_2$, we place $d_1$ on top of $d_2$, and identify the bottom vertices of $d_1$ with the top vertices of $d_2$.  As with Temperley-Lieb daigrams, closed loops are discarded.  For example,\\

%%% DIAGRAM MULT commented out%%%%%
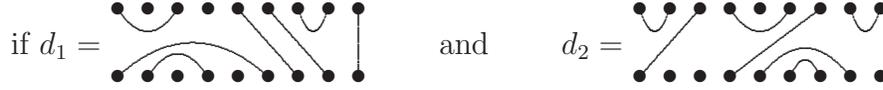
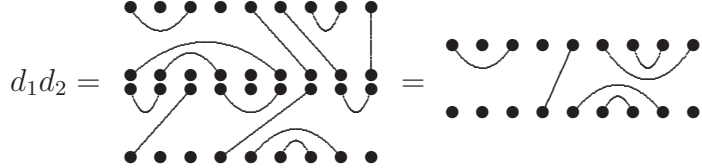
$$
\textrm{if}\  d_1 =
{\beginpicture
\setcoordinatesystem units <0.4cm,0.3cm>
\setplotarea x from 1 to 9, y from -1 to 2
\put{$\bullet$} at  1 -1  \put{$\bullet$} at  1 2
\put{$\bullet$} at  2 -1  \put{$\bullet$} at  2 2
\put{$\bullet$} at  3 -1  \put{$\bullet$} at  3 2
\put{$\bullet$} at  4 -1  \put{$\bullet$} at  4 2
\put{$\bullet$} at  5 -1  \put{$\bullet$} at  5 2
\put{$\bullet$} at  6 -1  \put{$\bullet$} at  6 2
\put{$\bullet$} at  7 -1  \put{$\bullet$} at  7 2
\put{$\bullet$} at  8 -1  \put{$\bullet$} at  8 2
\put{$\bullet$} at  9 -1  \put{$\bullet$} at  9 2
\plot 5 2  7 -1 /
\plot 6 2 8 -1 /
\plot 9 2 9 -1 /
\setquadratic
\plot 1 2 2 1 3 2 /
\plot 7 2 7.5 1 8 2 /
\plot 1 -1 3.5 .5 6 -1 /
\plot 2 -1 3 0 4 -1 /
\endpicture}
\qquad\hbox{ and }\qquad
d_2 = {\beginpicture
\setcoordinatesystem units <0.4cm,0.3cm>
\setplotarea x from 1 to 9, y from -1 to 2
\put{$\bullet$} at  1 -1  \put{$\bullet$} at  1 2
\put{$\bullet$} at  2 -1  \put{$\bullet$} at  2 2
\put{$\bullet$} at  3 -1  \put{$\bullet$} at  3 2
\put{$\bullet$} at  4 -1  \put{$\bullet$} at  4 2
\put{$\bullet$} at  5 -1  \put{$\bullet$} at  5 2
\put{$\bullet$} at  6 -1  \put{$\bullet$} at  6 2
\put{$\bullet$} at  7 -1  \put{$\bullet$} at  7 2
\put{$\bullet$} at  8 -1  \put{$\bullet$} at  8 2
\put{$\bullet$} at  9 -1  \put{$\bullet$} at  9 2
\plot 3 2  1 -1 /
\plot 7 2 4 -1 /
\setquadratic
\plot 1 2 1.5 1 2 2 /
\plot 4 2 5 1 6 2 /
\plot 8 2 8.5 1 9 2 /
\plot 5 -1 6.5 .25 8 -1 /
\plot 6 -1 6.5 -.25 7 -1 /
\endpicture}
$$
then
$$
d_1 d_2 =
\begin{array}{l}
{\beginpicture
\setcoordinatesystem units <0.4cm,0.3cm>
\setplotarea x from 1 to 9, y from -1 to 2
\put{$\bullet$} at  1 -1  \put{$\bullet$} at  1 2
\put{$\bullet$} at  2 -1  \put{$\bullet$} at  2 2
\put{$\bullet$} at  3 -1  \put{$\bullet$} at  3 2
\put{$\bullet$} at  4 -1  \put{$\bullet$} at  4 2
\put{$\bullet$} at  5 -1  \put{$\bullet$} at  5 2
\put{$\bullet$} at  6 -1  \put{$\bullet$} at  6 2
\put{$\bullet$} at  7 -1  \put{$\bullet$} at  7 2
\put{$\bullet$} at  8 -1  \put{$\bullet$} at  8 2
\put{$\bullet$} at  9 -1  \put{$\bullet$} at  9 2
\plot 5 2  7 -1 /
\plot 6 2 8 -1 /
\plot 9 2 9 -1 /
\setquadratic
\plot 1 2 2 1 3 2 /
\plot 7 2 7.5 1 8 2 /
\plot 1 -1 3.5 .5 6 -1 /
\plot 2 -1 3 0 4 -1 /
\endpicture}
\\
{\beginpicture
\setcoordinatesystem units <0.4cm,0.3cm>
\setplotarea x from 1 to 9, y from -1 to 2
\put{$\bullet$} at  1 -1  \put{$\bullet$} at  1 2
\put{$\bullet$} at  2 -1  \put{$\bullet$} at  2 2
\put{$\bullet$} at  3 -1  \put{$\bullet$} at  3 2
\put{$\bullet$} at  4 -1  \put{$\bullet$} at  4 2
\put{$\bullet$} at  5 -1  \put{$\bullet$} at  5 2
\put{$\bullet$} at  6 -1  \put{$\bullet$} at  6 2
\put{$\bullet$} at  7 -1  \put{$\bullet$} at  7 2
\put{$\bullet$} at  8 -1  \put{$\bullet$} at  8 2
\put{$\bullet$} at  9 -1  \put{$\bullet$} at  9 2
\plot 3 2  1 -1 /
\plot 7 2 4 -1 /
\setquadratic
\plot 1 2 1.5 1 2 2 /
\plot 4 2 5 1 6 2 /
\plot 8 2 8.5 1 9 2 /
\plot 5 -1 6.5 .25 8 -1 /
\plot 6 -1 6.5 -.25 7 -1 /
\endpicture}
\end{array}
= \
{\beginpicture
\setcoordinatesystem units <0.4cm,0.3cm>
\setplotarea x from 1 to 9, y from -1 to 2
\put{$\bullet$} at  1 -1  \put{$\bullet$} at  1 2
\put{$\bullet$} at  2 -1  \put{$\bullet$} at  2 2
\put{$\bullet$} at  3 -1  \put{$\bullet$} at  3 2
\put{$\bullet$} at  4 -1  \put{$\bullet$} at  4 2
\put{$\bullet$} at  5 -1  \put{$\bullet$} at  5 2
\put{$\bullet$} at  6 -1  \put{$\bullet$} at  6 2
\put{$\bullet$} at  7 -1  \put{$\bullet$} at  7 2
\put{$\bullet$} at  8 -1  \put{$\bullet$} at  8 2
\put{$\bullet$} at  9 -1  \put{$\bullet$} at  9 2
\plot 5 2  4 -1 /
\setquadratic
\plot 6 2 7.5 .5 9 2 /
\plot 1 2 2 1 3 2 /
\plot 7 2 7.5 1 8 2 /
\plot 5 -1 6.5 .25 8 -1 /
\plot 6 -1 6.5 -.25 7 -1 /
\endpicture}
$$

For any Motzkin diagram $d$, the symbols $\tau(d)$ and $\beta(d)$ refer to the sets of non-empty top and bottom vertices respectively, as in the Planar Rook Monoid.

\subsection{Decomposing the Motzkin Diagrams}

In \cite{Halverson}, it is proved that any diagram in the Motzkin monoid can be written in the form $d$=$rtl$ where $r \in RP_n$, $t \in TL_n$, and $l \in LP_n$. We now describe the algorithm to compute $r$, $t$, and $l$ for any given $d$. 

As in \cite{Kathy}, we let $r=R^S$ be the diagram with top set $S=\tau(d)$ and bottom set, $\{1,\ldots,|S|\}$. Let $l= L_T$ be the diagram with bottom set, $T=\beta(d)$ and top set, $\tau(l) =$ \{1,\ldots,$|T|$\}.
%There was a long section here about obtaining t using the sequence of 1's and -1's with pictures and everything. I feel bad erasing it, so I leave this comment as a memorial. --Eliezer

To obtain $t$, first shift the isolated vertices of $d$ to the right of the diagram, while preserving connections to the other vertices, to produce the diagram $d'$. For example, if $d$ is the diagram below, then $d'$ would be the diagram below it.

 \begin{center}
$d$\ =
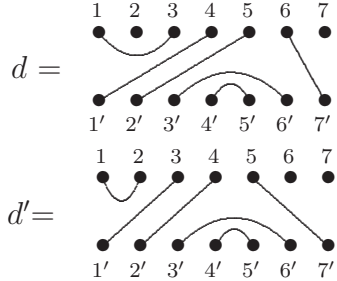
{\beginpicture
\setcoordinatesystem units <0.5cm,0.3cm>
\setplotarea x from 0 to 8, y from -1 to 2
\put{$\bullet$} at  1 -1  \put{$\bullet$} at  1 2
\put{$\bullet$} at  2 -1  \put{$\bullet$} at  2 2
\put{$\bullet$} at  3 -1  \put{$\bullet$} at  3 2
\put{$\bullet$} at  4 -1  \put{$\bullet$} at  4 2
\put{$\bullet$} at  5 -1  \put{$\bullet$} at  5 2
\put{$\bullet$} at  6 -1  \put{$\bullet$} at  6 2
\put{$\bullet$} at  7 -1  \put{$\bullet$} at  7 2
\plot  4 2  1 -1 /
\plot 5 2 2 -1 /
\plot 6 2 7 -1  /
\setquadratic
\plot 1 2 2 1 3 2 /
\plot 3 -1 4.5 .25 6 -1 /
\plot 4 -1 4.5 -.25 5 -1 /
\put{$\scriptstyle{1\,}$} at 1 3.0
\put{$\scriptstyle{2\,}$} at 2 3.0
\put{$\scriptstyle{3\,}$} at 3 3.0
\put{$\scriptstyle{4\,}$} at 4 3.0
\put{$\scriptstyle{5\,}$} at 5 3.0
\put{$\scriptstyle{6\,}$} at 6 3.0
\put{$\scriptstyle{7\,}$} at 7 3.0
\put{$\scriptstyle{1'\,}$} at 1 -2
\put{$\scriptstyle{2'\,}$} at 2 -2
\put{$\scriptstyle{3'\,}$} at 3 -2
\put{$\scriptstyle{4'\,}$} at 4 -2
\put{$\scriptstyle{5'\,}$} at 5 -2
\put{$\scriptstyle{6'\,}$} at 6 -2
\put{$\scriptstyle{7'\,}$} at 7 -2
\endpicture}\\
\end{center}

 \begin{center}
 $d'$= {\beginpicture
\setcoordinatesystem units <0.5cm,0.3cm>
\setplotarea x from 0 to 8, y from -1 to 2
\put{$\bullet$} at  1 -1  \put{$\bullet$} at  1 2
\put{$\bullet$} at  2 -1  \put{$\bullet$} at  2 2
\put{$\bullet$} at  3 -1  \put{$\bullet$} at  3 2
\put{$\bullet$} at  4 -1  \put{$\bullet$} at  4 2
\put{$\bullet$} at  5 -1  \put{$\bullet$} at  5 2
\put{$\bullet$} at  6 -1  \put{$\bullet$} at  6 2
\put{$\bullet$} at  7 -1  \put{$\bullet$} at  7 2
\plot  3 2  1 -1 /
\plot 4 2 2 -1 /
\plot 5 2 7 -1  /
\setquadratic
\plot 1 2 1.5 1 2 2 /
\plot 3 -1 4.5 .25 6 -1 /
\plot 4 -1 4.5 -.25 5 -1 /
\put{$\scriptstyle{1\,}$} at 1 3.0
\put{$\scriptstyle{2\,}$} at 2 3.0
\put{$\scriptstyle{3\,}$} at 3 3.0
\put{$\scriptstyle{4\,}$} at 4 3.0
\put{$\scriptstyle{5\,}$} at 5 3.0
\put{$\scriptstyle{6\,}$} at 6 3.0
\put{$\scriptstyle{7\,}$} at 7 3.0
\put{$\scriptstyle{1'\,}$} at 1 -2
\put{$\scriptstyle{2'\,}$} at 2 -2
\put{$\scriptstyle{3'\,}$} at 3 -2
\put{$\scriptstyle{4'\,}$} at 4 -2
\put{$\scriptstyle{5'\,}$} at 5 -2
\put{$\scriptstyle{6'\,}$} at 6 -2
\put{$\scriptstyle{7'\,}$} at 7 -2
\endpicture}
\end{center}

 Next we deal with the issue of filling in the empty vertices to complete the Temperley-Lieb diagram.  Let $\tau_{0}$ and $\beta_0$ denote the number of empty vertices on the top and bottom of $d'$ respectively.  For example, if $d'$ is as above, then $\tau_0 = 2$ and $\beta_0 = 0$.  Note that the difference between $\tau_{0}$ and $\beta_0$ must even.  We then have two cases for filling in the empty vertices:
\begin{itemize}
\item[] Case 1: $|\tau(d)| \leq |\beta(d)|$.
To turn $d'$ into a Temperley-Lieb diagram, first add horizontal edges on top from $j$ to $j+1$ for $j=|\tau(d)|+1, |\tau(d)|+3, \ldots, |\beta(d)|-1$.  Now, add vertical edges from $j$ to $j'$ for $j=|\beta(d)+1,\ldots,n.$
\item[]  Case 2: $|\tau(d)| \geq |\beta(d)|$.
To turn $d'$ into a Temperley-Lieb diagram, first add horizontal edges on bottom from $j'$ to $(j+1)'$ for $j'=|\beta(d)|+1, |\beta(d)|+3, \ldots,|\tau(d)|-1$.  Now, add vertical edges from $j$ to $j'$ for $j=|\tau(d)|+1,\ldots,n$
\end{itemize}
This gives a well defined way to decompose each diagram in the Motzkin Monoid.  Thus $M_n = RP_n\  TL_n\  LP_n$.

\subsection{Presentation of the Motzkin Monoid}
\begin{theorem}\label{mek} The Motzkin Monoid is generated by $t_i$, $r_i$ and $l_i$ subject to the following relations.

%\textcolor{red}
{
\begin{enumerate}
\item $r_i^3 = r_i^2 = l_i^2 = l_i^3$
\item $a)\ l_i l_{i+1} l_i = l_i l_{i+1} = l_{i+1} l_i l_{i+1} \\
b)\  r_i r_{i+1} r_i = r_{i+1} r_i = r_{i+1} r_i r_{i+1}$
\item $a)\  l_i r_i l_i = l_i $
\newline $b)\ r_i l_i r_i = r_i$
\item $a)\ l_{i+1} r_i l_i = l_{i+1}r_i$
\newline $b)\ r_{i-1} l_i r_i = r_{i-1} l_i$
\newline $c)\ r_i l_i r_{i+1} = l_i r_{i+1}$
\newline $d)\  l_i r_i l_{i-1} = r_i l_{i-1}$
\item $l_i r_i = r_{i+1} l_{i+1}$
\item $ \mathrm{If\ } |i - j | \geq 2, \mathrm{then}\  r_i l_j = r_j l_i,\  r_i r_j = r_j r_i,\  l_i l_j = l_j l_i,\  t_ir_j=r_jt_i,\  t_il_j=l_jt_i,\ t_j t_i=t_jt_i$
\item $t_i^2 =  t_i$
\item $\mathrm{If}\  |i-j|=1,\ \mathrm{then}\   t_it_jt_i= t_i\ $
\item $a)\ t_i l_i=t_i r_i\\
b)\ l_it_i=r_it_i$  
\item $a)\  t_i r_{i+1} = t_it_{i+1}l_i$
\newline $b)\  l_{i+1}t_i = r_it_{i+1}t_i$
\item a)\ $r_i r_{i+1} t_i = t_{i+1} r_i r_{i+1}\\
b)\ t_i l_{i+1} l_i = l_{i+1} l_i t_{i+1}$
\item $t_i l_i t_i = t_i$
\end{enumerate}
}%end red
\end{theorem}

Each of these relations can be easily verified by drawing the diagrams in question. For example the picture below justifies the relation $t_i r_{i+1} = t_i t_{i+1} l_i$.
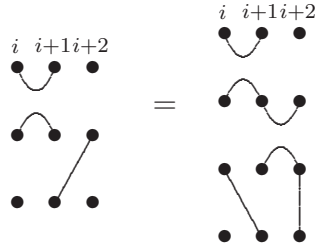
$$
{\beginpicture
\setcoordinatesystem units <0.5cm,0.3cm>
\setplotarea x from 1 to 3, y from -1 to 2
\put{$\bullet$} at  1 -4   \put{$\bullet$} at  1 -1  \put{$\bullet$} at  1 2 % (1)
\put{$\bullet$} at  2 -4   \put{$\bullet$} at  2 -1  \put{$\bullet$} at  2 2
\put{$\bullet$} at  3 -4   \put{$\bullet$} at  3 -1  \put{$\bullet$} at  3 2

%\plot  3 2  1 -1 / %(2)
\plot 2 -4 3 -1 /
\setquadratic
\plot 1 2 1.5 1 2 2 /  % (3)
\plot 1 -1 1.5 0 2 -1 /

\put{$\scriptstyle{\small{i}\,}$} at 1 3
\put{$\scriptstyle{\small{i+1}\,}$} at 2 3
\put{$\scriptstyle{\small{i+2}\,}$} at 3 3

\endpicture}
\quad = \quad
{\beginpicture
\setcoordinatesystem units <0.5cm,0.3cm>
\setplotarea x from 1 to 3, y from 2 to 4
\put{$\bullet$} at  1 -5.5   \put{$\bullet$} at  1 -2.5   \put{$\bullet$} at  1 0.5  \put{$\bullet$} at  1 3.5 % (1)
\put{$\bullet$} at  2 -5.5   \put{$\bullet$} at  2 -2.5   \put{$\bullet$} at  2 0.5  \put{$\bullet$} at  2 3.5
\put{$\bullet$} at  3 -5.5   \put{$\bullet$} at  3 -2.5   \put{$\bullet$} at  3 0.5  \put{$\bullet$} at  3 3.5

%\plot  3 2  1 -1 / %(2)
\plot 1 -2.5 2 -5.5 /
\plot 3 -2.5 3 -5.5 /
\setquadratic
\plot 1 3.5 1.5 2.5 2 3.5 /
\plot 1 0.5 1.5 1.5 2 0.5 /
\plot 2 0.5 2.5 -0.5 3 0.5 /
\plot 2 -2.5 2.5 -1.5 3 -2.5 /

\put{$\scriptstyle{\small{i}\,}$} at 1 4.5
\put{$\scriptstyle{\small{i+1}\,}$} at 2 4.5
\put{$\scriptstyle{\small{i+2}\,}$} at 3 4.5

\endpicture}
$$
%%End Kris' Section; Begin Eliezer's section
%%%
\subsection{The Motzkin Monoid as Words}
In this section we prove theorem \ref{mek}.  To start we will define a set of formal words which are subject exactly to the relations in theorem \ref{mek}.  Next we want to show that there is an isomorphism between this set and $M_n$.  

We define ${M_n'}$ to be the monoid generated by ${t_i}$, ${l_i}$, and ${r_i}$ subject to the same relations as above, and we let $\phi: {M_n'} \rightarrow M_n$ be the map given by
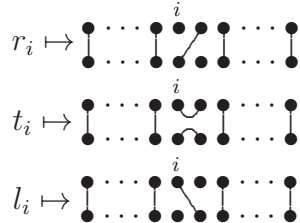
\begin{center}
${\beginpicture
\setcoordinatesystem units <0.3cm,0.15cm>
\setplotarea x from 1 to 3.8, y from 0 to 3
\linethickness=0.3pt
%\put{$\bullet$} at 1 -1 \put{$\bullet$} at 1 2
\put{$\bullet$} at 2 -1 \put{$\bullet$} at 2 2
%\put{$\bullet$} at 4 -1 \put{$\bullet$} at 4 2
\put{$\bullet$} at 5 -1 \put{$\bullet$} at 5 2
\put{$\bullet$} at 6 -1 \put{$\bullet$} at 6 2
\put{$\bullet$} at 7 -1 \put{$\bullet$} at 7 2
\put{$\bullet$} at 8 -1 \put{$\bullet$} at 8 2

\put{$\bullet$} at 11 -1 \put{$\bullet$} at 11 2
\put{$\cdots$} at 3.625 2 \put{$\cdots$} at 3.625 -1
\put{$\cdots$} at 9.625 2
\put{$\cdots$} at 9.625 -1

%\plot 1 2 1 -1 /
\plot 2 2 2 -1 /
\plot 5 2 5 -1 /
\plot 8 2 8 -1 /
\plot 6 -1 7 2 /
\plot 11 2 11 -1 /

\put{$\scriptstyle{i\,}$} at 6 3.75
\put{$r_i \mapsto$} at 0 .5
\endpicture} $
\end{center}

\begin{center}
${\beginpicture
\setcoordinatesystem units <0.3cm,0.15cm>
\setplotarea x from 1 to 3.8, y from 0 to 3
\linethickness=0.3pt
%\put{$\bullet$} at 1 -1 \put{$\bullet$} at 1 2
\put{$\bullet$} at 2 -1 \put{$\bullet$} at 2 2
%\put{$\bullet$} at 4 -1 \put{$\bullet$} at 4 2
\put{$\bullet$} at 5 -1 \put{$\bullet$} at 5 2
\put{$\bullet$} at 6 -1 \put{$\bullet$} at 6 2
\put{$\bullet$} at 7 -1 \put{$\bullet$} at 7 2
\put{$\bullet$} at 8 -1 \put{$\bullet$} at 8 2

\put{$\bullet$} at 11 -1 \put{$\bullet$} at 11 2
\put{$\cdots$} at 3.625 2 \put{$\cdots$} at 3.625 -1
\put{$\cdots$} at 9.625 2
\put{$\cdots$} at 9.625 -1

%\plot 1 2 1 -1 /
\plot 2 2 2 -1 /
\plot 5 2 5 -1 /
\plot 8 2 8 -1 /
\plot 11 2 11 -1 /
\setquadratic
\plot 6 2 6.5 1 7 2 /
\plot 6 -1 6.5 0 7 -1 /

\put{$\scriptstyle{i\,}$} at 6 3.75
\put{$t_i \mapsto$} at 0 .5
\endpicture} $
\end{center}
\begin{center}
$
{\beginpicture
\setcoordinatesystem units <0.3cm,0.15cm>
\setplotarea x from 1 to 3.8, y from 0 to 3
\linethickness=0.3pt
%\put{$\bullet$} at 1 -1 \put{$\bullet$} at 1 2
\put{$\bullet$} at 2 -1 \put{$\bullet$} at 2 2
%\put{$\bullet$} at 4 -1 \put{$\bullet$} at 4 2
\put{$\bullet$} at 5 -1 \put{$\bullet$} at 5 2
\put{$\bullet$} at 6 -1 \put{$\bullet$} at 6 2
\put{$\bullet$} at 7 -1 \put{$\bullet$} at 7 2
\put{$\bullet$} at 8 -1 \put{$\bullet$} at 8 2

\put{$\bullet$} at 11 -1 \put{$\bullet$} at 11 2
\put{$\cdots$} at 3.625 2 \put{$\cdots$} at 3.625 -1
\put{$\cdots$} at 9.625 2
\put{$\cdots$} at 9.625 -1

%\plot 1 2 1 -1 /
\plot 2 2 2 -1 /
\plot 5 2 5 -1 /
\plot 8 2 8 -1 /
\plot 6 2 7 -1 /
\plot 11 2 11 -1 /

\put{$\scriptstyle{i\,}$} at 6 3.75
\put{$l_i \mapsto$} at 0 .5
\endpicture} $
\end{center}

\noindent Since the diagrams $r_i, t_i,$ and $l_i$ satisfy the relations on ${M_n}$, the operations are preserved; thus $\phi$ is a homomorphism.  In order to prove that $\phi$ is also an isomorphism we want to show that ${M_n'}$ has at most as many distinct elements as there are diagrams in $M_n$. We begin by defining a standard word in ${M_n'}$.

\subsubsection{Standard Word}
Given a word $d$ in $M_n'$, let $\tau(d) = \{a_1,a_2,\ldots,a_k\}$ and $\beta(d) = \{b_1,b_2,\ldots,b_j\}$.  We say $d$ is in \textit{standard form} if $d=RTL$ where
$$R=\prod_{i=1}^{k} R_i^{a_i} \cdot \prod_{i=k+1}^{n} p_i\ \textrm{and}\ L=\prod_{i=j+1}^{n} p_i \cdot \prod_{i=0}^{j-1} L_{b_{j-i}}^{j-i}\,.$$\\
We then have three possibilities for $T$:
\begin{itemize}
\item If $k>j$ then, $T=T_1 t_{j+1} t_{j+3} \cdots t_{k-1}\ \textrm{where}\ T_1 \in \langle t_1, t_2, \ldots, t_{k-1} \rangle$.
\item If $k<j$ then, $T= t_{k+1} t_{k+3} \cdots t_{j-1}T_1\ \textrm{where}\ T_1 \in \langle t_1, t_2, \ldots, t_{j-1} \rangle$.
\item If $j=k$ then, $T=T_1\ \textrm{where}\ T_1 \in \langle t_1, t_2, \ldots, t_{j-1} \rangle$.
\end{itemize}
In all the three cases, $T$ is in standard form for Temperley-Lieb words.

This definition of a standard word corresponds to the diagram decomposition described by Halverson in \cite{halverson}.  By the definition of $\phi$, ${M_n'}$ maps to $M_n$.  So if we show that we can standardize any formal word in ${M_n'}$ then we will have proven theorem \ref{mek}.

The following relations can be derived from the relations on ${M_n'}$ and will be useful:
%\textcolor{red}{
\begin{itemize}
\item $r_i l_i= l_{i-1}r_{i-1}=p_i$
\item $r_i = p_i r_i=r_i p_{i+1}$
\item $l_i = p_{i+1} l_i = l_i p_i$
\item $p_i l_i= p_{i+1}r_i=p_i p_{i+1}$
\item $t_ir_i=t_il_i=t_ip_i=t_ip_{i+1}=t_ip_ip_{i+1}$
\item $r_it_i=l_it_i=p_it_i=p_{i+1}t_i=p_ip_{i+1}t_i$
\end{itemize}
}%end red
In order to prove that every element of ${M_n'}$ is equivalent to a standard word, note that the identity is a standard word.  Next, given a standard word $W$ and a generator $x_i \in \{ r_i$, $ l_i$, $ t_i\}$ we show that our relations are sufficient standardize the product $Wx_i$.  In order to do this we proceed in three steps: First we get the word $Wx_i$ into the form $P_1TP_2$, known as $PTP$ form, where $P_1, P_2 \in P_n$ and $T \in TL_n$.  Next we manipulate $PTP$ form in order to obtain minimal $RTL$ form where $R \in RP_n, L \in LP_n, T \in TL_n$.  After this, we manipulate the Temperley-Lieb diagram to obtain standard form.\\

\subsubsection{PTP form}
\begin{lemma}\label{rtlx to ptp} If $RTL$ is a word in standard form and $x_i \in \{r_i, l_i, t_i\}$, then $RTL x_i$ is equal to a word in $PTP$ form.
\end{lemma}

\begin{proof}
Notice that appending an $ r_i$ or an $ l_i$ yields a word already in $PTP$ form, so it suffices to show that  $RTL\cdot t_i$ is equivalent to a word in $PTP$ form.  We will proceed using reverse induction on $i$.\\

For the base case we append $t_{n-1}$.  There are several cases to consider depending on what the bottom set of $L$ is:\\

%\textbf{\textcolor{red}{Begin Megan}}
Case 1: $(n-1)', n' \in \beta (L)$.\\
Because $d$ is in standard form, there exist integers $j$ and $k$, with $j\leq i$ and $j\leq k \leq n$ such that $L = L^{k,b_k}L^{k-1,b_{k-1}}\dots L^{j+1,i+1}L^{j,i}\dots L^{1,b_1}$. For simplicity, let $L_0=L^{k,b_k}L^{k-1,b_{k-1}}\dots L^{j+2,b_{j+2}}$. Let $L_1=L^{j-1,b_{j-1}}L^{j-2,b_{j-2}}\dots L^{1,b_1}.$
Thus $$L=L_0 (l_{j+1} l_{j+2} \dots l_i) (l_j l_{j+1} \dots l_{i-1})L_1.$$
By the relation $l_pl_q=l_ql_p$ for $|p-q|\geq2$, we rearrange the terms to obtain
$L=L_0 (l_{j+1} l_j) (l_{j+2}l_{j+1})\dots (l_i l_{i-1}) L_1$.

The indices of all the components of $L_1$ are strictly less than $b_{j-1}$, which is in turn strictly less than $i$, implying that the highest index of any of $L_1$'s factors is less than or equal to $i-2$. Thus, $t_i$ can commute past all these elements. Thus,
$$Lt_i=L_0(l_{j+1} l_j) (l_{j+2}l_{j+1})\dots (l_i l_{i-1}) t_i L_1.$$
By repeatedly applying the relation $l_{p+1}l_pt_{p+1}=t_pl_{p+1}l_p$ we obtain
\begin{eqnarray*}
Lt_i&=&L_0(l_{j+1} l_j) (l_{j+2}l_{j+1})\dots l_{i-1} l_{i-2} t_{i-1}(l_i l_{i-1}) L_1\\
&=&L_0(l_{j+1} l_j) (l_{j+2}l_{j+1})\dots (l_{i-2} l_{i-3}) t_{i-2} (l_{i-1} l_{i-2}) (l_i l_{i-1})L_1\\
&\vdots&\\
&=&L_0 t_j(l_{j+1} l_j) (l_{j+2}l_{j+1})\dots (l_i l_{i-1})L_1
\end{eqnarray*}
All the factors of $L_1$ have index greater than or equal to $j+2$, implying that $t_j$ commutes with $L_1$. Therefore $Lt_i=t_jL$. Thus $dt_i=RTLt_i=RT'L$.\\

Case 2: $(n-1)', n' \not\in \beta( L)$.

Since the $(n-1)'$ and $n'$ vertices are empty, we have that $ L= L_0  p_{n-1} p_n$ where $ L_0 \in\  \langle l_1, \ldots, l_{n-3}, p_1, \ldots, p_{n-2}\rangle$.
We have then,

$$R T L  t_{n-1} = R T L_0 p_{n-1} p_n t_{n-1} = R T p_{n-1} p_n t_{n-1} L_0.$$

Now to get to $PTP$ form we need to move $p_{n-1} p_n$.  Note $n',(n-1)' \notin \beta(L)$, so $n, n-1 \notin \tau(L)$.  Since $RTL$ is in standard form, this tells us something about $T$.  There are three subcases depending on $\beta(R)$.

Subcase 1: ${(n-1)', n' \not\in\beta(R)}$.
%If ${n-1, n \not\in\beta(R)}$ then $R = Rp_{n-1} p_n$ as in the following diagram:
%$${\beginpicture
%\setcoordinatesystem units <0.5cm,0.3cm>
%\setplotarea x from 0 to 8, y from -1 to 2
%\put{Subcase 1 Diagram} at  3 4
%\put{$\begin{picture}(50,50)
%\put(0,0){\framebox(65,20){$R$}}
%\end{picture}$} at 2.5 3
%\put{$\bullet$} at 1 -3  \put{$\bullet$} at  1 0
%\put{$\cdots$} at 2 -1.5
%\put{$\begin{picture}(50,50)
%\put(0,0){\framebox(40,25){$T$}}
%\end{picture}$} at 2.5 -3
%\put{$\begin{picture}(50,34)
%\put(0,0){\framebox(40,28){$L$}}
%\end{picture}$} at 2.5 -7.25
%\put{$\bullet$} at 3 -3 \put{$\bullet$} at  3 0
%\put{$\bullet$} at 4 -3 \put{$\bullet$} at  4 0  \put{$\bullet$} at 4 -6  \put{$\bullet$} at 4 -9
%\put{$\bullet$} at 5 -3 \put{$\bullet$} at  5 0  \put{$\bullet$} at 5 -6  \put{$\bullet$} at 5 -9

%\plot 1 0 1 -3 /
%\plot 3 0 3 -3 /
%\plot 4 -3 4 -6 /
%\plot 5 -3 5 -6 /
%\setquadratic

%\put{$\scriptstyle{i\,}$} at 1 -5
%\put{$\scriptstyle{i+1\,}$} at 2 -5
%\endpicture}$$

Then $T \in\  \langle t_1, \dots, t_{n-3} \rangle.$  In this case every element of $T$ commutes with $p_{n-1} p_n t_{n-1}$, and so $$RTLt_{n-1}= RT p_{n-1} p_n t_{n-1}L_0 = R p_{n-1} p_n (t_{n-1}T) L_0 $$

Subcase 2: ${(n-1)',n' \in\beta(R)}$.

Then $T$ can be written in the form $T_0t_{n-1}$, where $T_0 \in \langle t_1,\ldots,t_{n-2} \rangle$.  So $T= T_0 t_{n-1}.$  Then $$RTL t_{n-1} = R T_0 t_{n-1} p_{n-1} p_n t_{n-1} L_0 = RT_0 t_{n-1} L_0$$
%$${\beginpicture
%\setcoordinatesystem units <0.5cm,0.3cm>
%\setplotarea x from 0 to 8, y from -1 to 2
%\put{Case 2 Diagram} at 3 4
%\put{$\begin{picture}(50,50)
%\put(0,0){\framebox(65,20){$R$}}
%\end{picture}$} at 2.5 3
%\put{$\begin{picture}(50,50)
%\put(0,0){\framebox(65,20){$T_o$}}
%\end{picture}$} at 2.5 0.6
%\put{$\begin{picture}(50,34)
%\put(0,0){\framebox(40,29){$L$}}
%\end{picture}$} at 2.5 -6.8

%\put{$\bullet$} at 1 -2.4 \put{$\bullet$} at 1 -5.4
%\put{$\cdots$} at 2 -3.9
%\put{$\bullet$} at 3 -2.4 \put{$\bullet$} at 3 -5.4
%\put{$\bullet$} at 4 -2.4 \put{$\bullet$} at 4 -5.4 \put{$\bullet$} at 4 -8.4
%\put{$\bullet$} at 5 -2.4 \put{$\bullet$} at 5 -5.4 \put{$\bullet$} at 5 -8.4

%\plot 1 -2.4 1 -5.4 /
%\plot 3 -2.4 3 -5.4 /
%\plot 4 -5.4 4 -8.4 /
%\plot 5 -5.4 5 -8.4 /
%\setquadratic
%\plot 4 -2.4 4.5 -3.4 5 -2.4 /
%\plot 4 -5.4 4.5 -4.4 5 -5.4 /

%\put{$\scriptstyle{i\,}$} at 1 -5
%\put{$\scriptstyle{i+1\,}$} at 2 -5
%\endpicture}$$
%\vspace{3mm}

Subcase 3: ${(n-1)' \in\beta(R), n' \not\in\beta(R)}$.

Then $T$ can be written as $T=T_0t_{n-2}$, where $T_0 \in \langle t_1,\ldots,t_{n-2} \rangle$. By our relations we get, $$RTLt_{n-1}=RT_0 t_{n-2} p_{n-1} p_n t_{n-1} L_0= RT_0 t_{n-2} l_{n-1} l_{n-2} L_0$$

%$${\beginpicture
%\setcoordinatesystem units <0.5cm,0.3cm>
%\setplotarea x from 0 to 8, y from -1 to 2
%\put{Case 3 Diagram} at 3.5 4
%\put{$\begin{picture}(50,50)
%\put(0,0){\framebox(75,20){$R$}}
%\end{picture}$} at 2.6 3
%\put{$\begin{picture}(50,50)
%\put(0,0){\framebox(58.5,20){$T$}}
%\end{picture}$} at 2.7 -2.5
%\put{$\begin{picture}(50,34)
%\put(0,0){\framebox(37,28){$L$}}
%\end{picture}$} at 2.75 -9.6
%\put{$\bullet$} at 1 -3  \put{$\bullet$} at  1 0 \put{$\bullet$} at  1 -5.35 \put{$\bullet$} at  1 -8.35
%\put{$\cdots$} at 2 -6.85
%\put{$\cdots$} at 3 -1.5 \put{$\bullet$} at  3 -5.35 \put{$\bullet$} at  3 -8.35
%\put{$\bullet$} at  4 -5.35 \put{$\bullet$} at  4 -8.35 \put{$\bullet$} at  4 -11.35
%\put{$\bullet$} at 5 -3 \put{$\bullet$} at  5 0  \put{$\bullet$} at  5 -5.35 \put{$\bullet$} at  5 -8.35 \put{$%\bullet$} at  5 -11.35
%\put{$\bullet$} at 6 -3 \put{$\bullet$} at  6 0  \put{$\bullet$} at 6 -5.35 \put{$\bullet$} at  6 -8.35 \put{$%\bullet$} at  6 -11.35

%\plot 1 0 1 -3 /
%\plot 1 -5.35 1 -8.35 /
%\plot 3 -5.35 3 -8.35 /
%\plot 5 0 5 -3 /
%\plot 6 0 6 -3 /
%\plot 6 -5.35 6 -8.35 /

%\setquadratic
%\plot 4 -5.35 4.5 -6.35 5 -5.35 /
%\plot 4 -8.35 4.5 -7.35 5 -8.35 /

%%\put{$\scriptstyle{i\,}$} at 1 -5
%%\put{$\scriptstyle{i+1\,}$} at 2 -5
%\endpicture}$$
In all three subcases, $RTLt_{n-1}$ is equal to a word in $PTP$ form.\\

Case 3: $(n-1)'\in\beta( L), n'\not\in\beta(L)$.

Since the $n$th vertex in $ L$ is empty, $ L= L p_n$.  Using our relations, this implies that $R T L t_{n-1}= R T L p_n t_{n-1}= R T L p_{n-1} p_n t_{n-1}.$ Now we are reduced to case 2.\\

Case 4: $(n-1)'\not\in\beta( L), n'\in\beta( L)$.
This case can be reduced to case 2 using analogous reasoning.\\
%\textcolor{red}{\textbf{End Megan}}

In conclusion, we have shown that if $RTL \in M_n'$ is any word in standard form, then $RTLt_{n-1}$ is equivalent to a word in $PTP$ form. We now move on to the inductive step to show that this holds for any $t_i$.  Assume that for $j=i+1, i+2,...,n-1$ $ R T L t_j$ can be put into $PTP$ form. We now show that this assumption implies that $RTL t_i$ can be put into $PTP$ form. Consider the following three cases:\\

%\textcolor{red}{\textbf{Begin Megan}}
Case 1: $i', i'+1, i'+2 \notin\beta( L)$

Since $i, i+1, i+2 \notin\beta( L)$ we can write $L$ as $ L= L p_i p_{i+1} p_{i+2}$. Thus,
\begin{eqnarray*}
R T L t_i &=& R T L p_i p_{i+1} p_{i+2} t_i\\
&=& R T L p_i p_{i+1} p_{i+2} t_{i+1} r_i  r_{i+1}\\
&=& ( R T L p_i p_{i+1} p_{i+2} t_{i+1}) r_i  r_{i+1}\\
&=& RTLt_{i+1}r_ir_{i+1}.
\end{eqnarray*}
The factor $RTLt_{i+1}$ right can be put into $PTP$ form, by the inductive hypothesis. Thus we have $RTLt_i=P_1T'P_2r_i r_{i+1}$, which is in $PTP$ form.\\
%Kris has made a diagram, illustrating the relation that justifies this.

Case 2: $i', i'+1 \in \beta (L)$

We know that $L$ can we written as $L= l_i l_{i-1} L_0$ where $r_0 \in \langle l_1 \ldots l_{i-2} \rangle$.  So,
\begin{eqnarray*}
RTL t_i &=& RTl_i l_{i-1} L_0 t_i\\
&=& RT l_i l_{i-1} t_i L_0\\
&=& RTl_i r_i t_{i-1} l_i l_{i-1} L_0\\
&=&RT p_{i+1} t_{i-1} l_i l_{i-1} L_0\\
&=&RT t_{i-1}p_{i+1} l_i l_{i-1} L_0.
\end{eqnarray*}

Case 3: ${i', i'+1 \notin\beta( L), i'+2 \in\beta( L)}$

Since $L$ is in standard form, and $i',i'+1 \notin\beta(L)$, we know that we cannot have a vertical line from $i+2$ to $i'+2$.  Furthermore, $i',i'+1 \notin\beta(L)$ allows us to write $L$ equivalently as $L= L_0  l_i  l_{i+1}$ (because necessarily $p_i, p_{i+1}$ are in our product $L_0$), where $ L_0$ is defined by $\tau( L_0)=\tau( L),$ and$ \beta( L_0)=(\beta( L)\setminus\{i+2\})\cup\{i\}$. So we have,
 \begin{eqnarray*}
 R T L t_i &=& R T L_0 l_i l_{i+1} t_i\\
 &=& ( R T L_0 p_{i+1} p_{i+2} t_{i+1}) t_i.
 \end{eqnarray*}
 The factor in parenthesis can be put into $P_1TP_2$ form, by the inductive hypothesis. Thus $ R T L t_i= P_1 T' P_2 t_i$. Note that $\beta(P_2)$ contains $i, i+1, \textrm{ and } i+2$. So $P_1T'P_2 t_i$ can be put into standard form, by case 2.\\

Case 4: $i\in\beta(L), i+1\notin\beta(L)$

Note that in this case $L=Lp_{i+1}$.  Using our relations we have, $RTLt_i=RTLp_{i+1}t_i=RTLp_{i+1}p_it_i$. Note that $i$ and $i+1$ are not in $\beta(Lp_{i+1}p_i)$, implying that this word can be standardized by the previous case.\\

Case 5: $i \notin\beta(L), i+1\in\beta(L)$

This case is analogous to case 4.\\
\end{proof}
%\textcolor{red}{\textbf{End Megan}}

\subsubsection{Minimal $RTL$ form}

Rewriting a word in $PTP$ form as a word in standard form is analogous to taking a Motzkin diagram drawn as the product of a Temperley-Lieb diagram and two planar rook diagrams and redrawing it so that the empty vertices in the bottom of the top diagram and the empty vertices in the top of the bottom diagram are all on the far right. This motivates the following lemmas, each of which concerns words in $M_n'$ but corresponds to shifting empty vertices rightward in diagrams in $M_n$.

\begin{lemma}[Hop]\label{hopping}
Let $i$ and $k$ be natural numbers.  If $i<k$ and $k-i$ is even then, 
$$ (t_{k-1} t_{k-3} \ldots t_{i+1}) p_i =  (t_{k-1} t_{k-3} \ldots t_{i+1}) (t_{k-2} t_{k-4} \ldots t_{i})(l_{k-1}l_{k-2} \ldots l_i).$$
\end{lemma}

Intuitively we are ``hopping" a dead end in the $i$th vertex over a series of horizontal edges to the $k$th vertex as shown below: 
$$
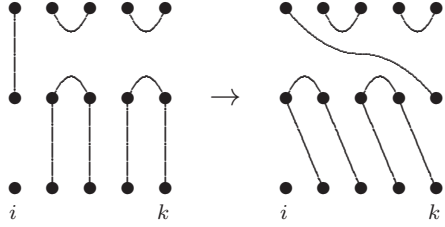
{\beginpicture
\setcoordinatesystem units <0.5cm,0.3cm>
\setplotarea x from 1 to 8, y from 0 to 2
\put{$\rightarrow$} at 6.6 -3
\put{$\bullet$} at 1 1 % (1)
\put{$\bullet$} at 2 1
\put{$\bullet$} at 3 1
\put{$\bullet$} at 4 1
\put{$\bullet$} at 5 1
\put{$\bullet$} at 1 -3
\put{$\bullet$} at 2 -3
\put{$\bullet$} at 3 -3
\put{$\bullet$} at 4 -3
\put{$\bullet$} at 5 -3
\put{$\bullet$} at 1 -7
\put{$\bullet$} at 2 -7
\put{$\bullet$} at 3 -7
\put{$\bullet$} at 4 -7
\put{$\bullet$} at 5 -7
%\put{$\framebox[0.9 in]{$T_0$}$} at 4 0.45
%\put{$\cdots$} at 4 2
%\put{$\cdots$} at 4 -3
%\put{$\cdots$} at 4 -7
\plot 1 1 1 -3 /
\plot 2 -3 2 -7 /
\plot 3 -3 3 -7 /
\plot 4 -3 4 -7 /
\plot 5 -3 5 -7 /
\setquadratic
\plot 2 -3 2.5 -2 3 -3 /
\plot 4 -3 4.5 -2 5 -3 /
\plot 2 1 2.5 0 3 1 /
\plot 4 1 4.5 0 5 1 /
\put{$\scriptstyle{i\,}$} at 1 -8
\put{$\scriptstyle{k\,}$} at 5 -8
\endpicture}
{\beginpicture
\setcoordinatesystem units <0.5cm,0.3cm>
\setplotarea x from 1 to 5, y from 0 to 2
\put{$\bullet$} at 1 1 % (1)
\put{$\bullet$} at 2 1
\put{$\bullet$} at 3 1
\put{$\bullet$} at 4 1
\put{$\bullet$} at 5 1
\put{$\bullet$} at 1 -3
\put{$\bullet$} at 2 -3
\put{$\bullet$} at 3 -3
\put{$\bullet$} at 4 -3
\put{$\bullet$} at 5 -3
%\put{$\framebox[0.9 in]{$T_0$}$} at 4 0.45
%\put{$\cdots$} at 4 2
%\put{$\cdots$} at 4 -3
%\put{$\cdots$} at 5 -7
\put{$\scriptstyle{i\,}$} at 1 -8
\put{$\scriptstyle{k\,}$} at 5 -8
\plot 1 -3 2 -7 /
\plot 2 -3 3 -7 /
\plot 3 -3 4 -7 /
\plot 4 -3 5 -7 /
\setquadratic
\plot 1 1 2 -.5 3 -1 /
\plot 3 -1 4 -1.5 5 -3 /
\plot 1 -3 1.5 -2 2 -3 /
\plot 3 -3 3.5 -2 4 -3 /
\plot 2 1 2.5 0 3 1 /
\plot 4 1 4.5 0 5 1 /
\put{$\bullet$} at  1 -7
\put{$\bullet$} at  2 -7
\put{$\bullet$} at  3 -7
\put{$\bullet$} at  4 -7
\put{$\bullet$} at  5 -7
\endpicture}
$$

\begin{proof} We prove Lemma \ref{hopping} by induction on $k$.\\ 

\noindent The case $k=i+2$ is simply the relation $t_{i+1}p_i = t_{i+1}t_il_{i+1} l_i$.\\ 

%\textcolor{red}{\textbf{Begin Megan}}
\noindent Assume the result holds for $k-2$: $$(t_{k-3} t_{k-5} \ldots t_{i+1}) p_i =  (t_{k-3} t_{k-5} \ldots t_{i+1}) (t_{k-4} t_{k-6} \ldots t_{i})(l_{k-3}l_{k-4} \ldots l_i)$$
Then \begin{eqnarray*}
&& (t_{k-1} t_{k-3} \ldots t_{i+3} t_{i+1}) p_i \\ 
&=& t_{k-1}(t_{k-3} t_{k-5} \ldots t_{i+1}) (t_{k-4} t_{k-6} \ldots t_{i})(l_{k-3}l_{k-4} \ldots l_i) \ (\textrm{inductive hypothesis}) \\ 
&=& t_{k-1} (t_{k-3} t_{k-5} \ldots t_{i+1}) (t_{k-4} t_{k-6} \ldots t_{i}) p_{k-2} (l_{k-3}l_{k-4} \ldots l_i) \ (l_i= p_{i+1}l_i) \\ 
&=& (t_{k-3} t_{k-5} \ldots t_{i+1}) (t_{k-4} t_{k-6} \ldots t_{i})t_{k-1} p_{k-2} (l_{k-3}l_{k-4} \ldots l_i) \ (\textrm{commuting}\  t_{k-1})\\ 
&=& (t_{k-3} t_{k-5} \ldots t_{i+1}) (t_{k-4} t_{k-6} \ldots t_{i})t_{k-1} t_{k-2} l_{k-1}l_{k-2}(l_{k-3}l_{k-4} \ldots l_i)\ (\textrm{base case}) \\ 
&=& (t_{k-1} t_{k-3} \ldots t_{i+1}) (t_{k-2} t_{k-4} \ldots t_{i})(l_{k-1}l_{k-2} \ldots l_i)\  (\textrm{commuting } t_{k-1} \textrm{ and } t_{k-2}).
\end{eqnarray*}
%\textcolor{red}{\textbf{End Megan}}

\end{proof}

\begin{lemma}[Slide]\label{slide}
If $T \in \langle t_{i+1}, t_{i+2} \ldots , t_{k-1} \rangle$, then 
$$(l_{k-1} l_{k-2} \ldots l_i) T = T' (l_{k-1} l_{k-2} \ldots l_i)$$
for some $T' \in \, \langle t_i, t_{i+1} \ldots t_{k-2} \rangle.$ 
\end{lemma}

Diagrammatically we are ``sliding" the product of $l$'s under the product of $t$'s, as shown below:

$$
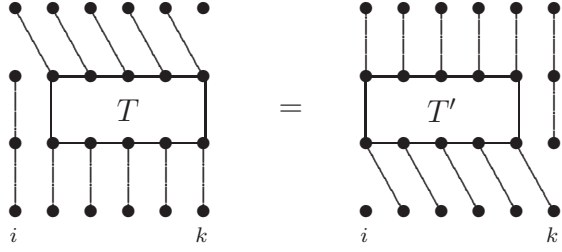
{\beginpicture
\setcoordinatesystem units <0.5cm,0.3cm>
\setplotarea x from 0 to 8, y from -1 to -1
\put{$\bullet$} at  1 2 % (1)
\put{$\bullet$} at  2 2
\put{$\bullet$} at  3 2
\put{$\bullet$} at  4 2
\put{$\bullet$} at 5 2
\put{$\begin{picture}(50,50)
\put(3,0){\framebox(58,25){$T$}}
\end{picture}$} at 3.5 -1
\put{$\bullet$} at 1 -7 \put{$\bullet$} at 1 -4 \put{$\bullet$} at  1 -1 % (1)
\put{$\bullet$} at 2 -7 \put{$\bullet$} at 2 -4 \put{$\bullet$} at  2 -1
\put{$\bullet$} at 3 -7 \put{$\bullet$} at 3 -4 \put{$\bullet$} at  3 -1
\put{$\bullet$} at 4 -7 \put{$\bullet$} at 4 -4 \put{$\bullet$} at 4 -1
\put{$\bullet$} at 5 -7 \put{$\bullet$} at 5 -4  \put{$\bullet$} at  5 -1
 \put{$\bullet$} at 6 -7 \put{$\bullet$} at 6 -4  \put{$\bullet$} at  6 -1 \put{$\bullet$} at  6 2
%\put{$\cdots$} at 3.5 0.5
\plot 2 -4 2 -7 /
\plot 3 -4 3 -7 /
\plot 4 -4 4 -7 /
\plot 5 -4 5 -7 /
\plot 6 -4 6 -7 /
\plot 1 -1 1 -7 /
\plot 1 2 2 -1 /
\plot 2 2 3 -1 /
\plot 3 2 4 -1 /
\plot 4 2 5 -1 /
\plot 5 2 6 -1 /
\setquadratic

\put{$\scriptstyle{i\,}$} at 1 -8 % (4)
%\put{$\scriptstyle{\ldots\,}$} at 3.5 -8
\put{$\scriptstyle{k\,}$} at 6 -8
\endpicture}
{\beginpicture
\setcoordinatesystem units <0.5cm,0.3cm>
\setplotarea x from 0 to 8, y from -1 to -1
\put{$\bullet$} at  1 2 % (1)
\put{$\bullet$} at  2 2
\put{$\bullet$} at  4 2
\put{$\bullet$} at 5 2
\put{$\begin{picture}(50,50)
\put(1.5,0){\framebox(57.5,25){$T'$}}
\end{picture}$} at 2.65 -1
\put{$=$} at -1 -2.5
\put{$\bullet$} at 1 -7 \put{$\bullet$} at 1 -4 \put{$\bullet$} at  1 -1 % (1)
\put{$\bullet$} at 2 -7 \put{$\bullet$} at 2 -4 \put{$\bullet$} at  2 -1
\put{$\bullet$} at 3 -7 \put{$\bullet$} at 3 2 \put{$\bullet$} at  3 -1  \put{$\bullet$} at 3 -4
\put{$\bullet$} at 4 -1 \put{$\bullet$} at 4 -4 \put{$\bullet$} at 4 -7
\put{$\bullet$} at 5 -7 \put{$\bullet$} at 5 -4  \put{$\bullet$} at  5 -1
 \put{$\bullet$} at 6 -7 \put{$\bullet$} at 6 -4  \put{$\bullet$} at  6 -1 \put{$\bullet$} at  6 2
%\put{$\cdots$} at 3.5 -5.5
\plot 1 2 1 -1 /
\plot 2 2 2 -1 /
\plot 3 2 3 -1 /
\plot 4 2 4 -1 /
\plot 5 2 5 -1 /
\plot 6 2 6 -4 /
\plot 1 -4 2 -7 /
\plot 2 -4 3 -7 /
\plot 3 -4 4 -7 /
\plot 4 -4 5 -7 /
\plot 5 -4 6 -7 /

\setquadratic

\put{$\scriptstyle{i\,}$} at 1 -8 % (4)
%\put{$\scriptstyle{\ldots\,}$} at 3.5 -8
\put{$\scriptstyle{k\,}$} at 6 -8
\endpicture}
$$

\begin{proof}
We will prove lemma \ref{slide} using induction on the length of the word $T$.  \\

\noindent Base case: For our base case we want to show that $l_{k-1} \ldots l_i t_j = t_{j-1} l_{k-1} \ldots l_i$ for any $t_j \in TL_n$.  Since $j \in [i+1, k-1]$ then there exists an $l_{j-1}$ in our product of $l$'s.  We will commute $t_j$ up to this point and then use our relations as follows:

%\textcolor{red}{
\begin{eqnarray*}
l_{k-1} \ldots l_i t_j &=& l_{k-1} \ldots l_{j-1} \ldots l_i t_j\\ 
&=& l_{k-1} \ldots l_{j-2} l_{j-1} t_j l_{j-2} \ldots l_i \ (l_i t_j = t_j l_i\  \textrm{if}\  |i-j| \geq 2)\\
&=& l_{k-1} \ldots t_{j-1} l_{j-2} l_{j-1} l_{j-2} \ldots l_i \ (t_i l_{i+1} l_i = l_{i+1} l_i t_{i+1})\\
&=& t_{j-1} l_{k-1} \ldots l_i\ (l_i t_j = t_j l_i\  \textrm{if}\  |i-j| \geq 2).
\end{eqnarray*}
%}end red

\noindent Inductive step: Now, assume $l_{k-1} \ldots l_i  T = T' l_{k-1} \ldots l_i$ where $T$ and $T'$ are some products of $t$'s as described above.  If we append some $t_m \in TL_n$, we have that:

%\textcolor{red}{
\begin{eqnarray*}
l_{k-1} \ldots l_i T t_m &=& T' l_{k-1} \ldots l_i t_m \ (\textrm{inductive hypothesis}) \\ 
&=& T' t_{m-1} l_{k-1} \ldots l_{m-1} \ldots l_i\ (\textrm{base case})\\
&=& T'' l_{k-1} \ldots l_i
\end{eqnarray*}
%}%end red

where $T'' \in \, \langle t_i, t_{i+1} \ldots t_{k-2} \rangle.$ So by induction we have that $(l_{k-1} l_{k-2} \ldots l_i) T = T' (l_{k-1} l_{k-2} \ldots l_i)$ as desired.
\end{proof}
   
\begin{lemma}[Burrow]\label{burrow}
Let i be a natural number.  Then, $t_{i-1} p_i = t_{i-1} t_i l_{i-1} l_i.$
\end{lemma}

Intuitively, we will use this to allow a horizontal edge to ``burrow under" a neighboring edge. 

$$
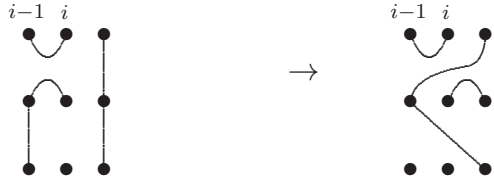
{\beginpicture
\setcoordinatesystem units <0.5cm,0.3cm>
\setplotarea x from 1 to 6, y from -4 to 2
\put{$\bullet$} at 1 -4  \put{$\bullet$} at  1 -1  \put{$\bullet$} at  1 2
\put{$\bullet$} at 2 -4  \put{$\bullet$} at  2 -1  \put{$\bullet$} at  2 2
\put{$\bullet$} at 3 -4  \put{$\bullet$} at  3 -1  \put{$\bullet$} at  3 2
%\plot 1 2 1 -1 /
\plot 1 -1 1 -4 /
%\plot 2 -1 2 -4 /
%\plot 2 2 2 -1 /
\plot 3 2 3 -1 /
\plot 3 -1 3 -4 /
\setquadratic
\plot 1 2 1.5 1 2 2 /  % (3)
\plot 1 -1 1.5 0 2 -1 /

\put{$\scriptstyle{i-1\,}$} at 1 3.0
\put{$\scriptstyle{i\,}$} at 2 3
\endpicture}
\qquad \rightarrow \qquad
{\beginpicture
\setcoordinatesystem units <0.5cm,0.3cm>
\setplotarea x from 1 to 6, y from -1 to 2
\put{$\bullet$} at 1 -4 \put{$\bullet$} at  1 -1  \put{$\bullet$} at  1 2 % (1)
\put{$\bullet$} at 2 -4 \put{$\bullet$} at  2 -1  \put{$\bullet$} at  2 2
\put{$\bullet$} at 3 -4 \put{$\bullet$} at  3 -1  \put{$\bullet$} at  3 2
\plot 1 -1 3 -4 /
\setquadratic
\plot 1 2 1.5 1 2 2 /  % (3)
\plot 2 -1 2.5 0 3 -1 /
\plot 1 -1 1.4 0 2.5 0.6 /
\plot 2.5 0.6 2.85 1 3 2 /

\put{$\scriptstyle{i-1\,}$} at 1 3.0
\put{$\scriptstyle{i\,}$} at 2 3
\endpicture}
$$

\begin{proof}
The proof of this lemma follows easily from the relations $p_i = r_i l_i$ and $t_i r_{i+1} = t_i t_{i+1} l_i$.  Thus, $t_{i-1} p_i = t_{i-1} r_i l_i = t_{i-1} t_i l_{i-1} l_i.$
\end{proof}

\begin{lemma}[Wallslide]\label{wallslide1}
Let $i$ be a natural number. Let $T\in \langle t_1,\dots,t_{i-1}\rangle$ and $T' \in \langle t_{i+2},\dots,t_{n-1}\rangle$ be arbitrary. Then $T T' p_i = r_i T  T' l_i$.
\end{lemma}

This is the case where we have a ``punctured" vertical line (an edge from vertex $i$ on top to vertex $i'$ on bottom), and want to move it over a non-punctured vertical line\----``a wall"\---- as in the diagram:

$$
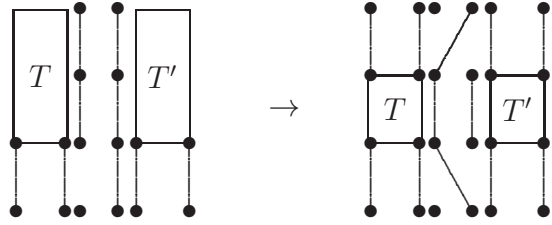
{\beginpicture
\setcoordinatesystem units <0.5cm,0.3cm>
\setplotarea x from 1 to 4, y from -1 to 2
\put{$\begin{picture}(50,50)
\put(0,0){\framebox(20,50){$T$}}
\end{picture}$} at 1 -1
\put{$\begin{picture}(50,50)
\put(0,0){\framebox(20,50){$T'$}}
\end{picture}$} at 4.25 -1
\put{$\bullet$} at 1 -7 \put{$\bullet$} at 1 -4 \put{$\bullet$} at  1 -1  \put{$\bullet$} at  1 2 % (1)
\put{$\bullet$} at 2 -7 \put{$\bullet$} at 2 -4 \put{$\bullet$} at  2 -1  \put{$\bullet$} at  2 2
\put{$\bullet$} at -.7 -7
\put{$\bullet$} at 3.9 -7
\put{$\bullet$} at 3.9 -4
\put{$\bullet$} at -.7 -4
\put{$\bullet$} at .6 -4
\put{$\bullet$} at 2.5 -4
\put{$\bullet$} at .6 -7
\put{$\bullet$} at 2.5 -7
%\put{$\cdots$} at 0 -5.5
%\put{$\cdots$} at 3.3 -5.5

\plot 2.5 -7 2.5 -4 /
\plot .6 -7 .6 -4 /
\plot -.7 -4 -.7 -7 /
\plot 3.9 -7 3.9 -4 /
\plot 2 2 2 -7 /
\plot 1 -1 1 -4 /
\plot 1 -1 1 2 /
\setquadratic

\endpicture}
{\beginpicture
\setcoordinatesystem units <0.5cm,0.3cm>
\setplotarea x from 1 to 4, y from -1 to 2
\put{$\begin{picture}(50,50)
\put(0,0){\framebox(20,25){$T$}}
\end{picture}$} at 1 -1
\put{$\begin{picture}(50,50)
\put(0,0){\framebox(20,25){$T'$}}
\end{picture}$} at 4.25 -1
\put{$\bullet$} at 1 -7 \put{$\bullet$} at 1 -4 \put{$\bullet$} at  1 -1  \put{$\bullet$} at  1 2 % (1)
\put{$\bullet$} at 2 -7 \put{$\bullet$} at 2 -4 \put{$\bullet$} at  2 -1  \put{$\bullet$} at  2 2
\put{$\rightarrow$} at -3 -2.5
\put{$\bullet$} at -.7 -7
\put{$\bullet$} at 3.9 -7
\put{$\bullet$} at 3.9 -4
\put{$\bullet$} at -.7 -4
\put{$\bullet$} at .6 -4
\put{$\bullet$} at 2.5 -4
\put{$\bullet$} at .6 -7
\put{$\bullet$} at 2.5 -7
\put{$\bullet$} at -.7 2
\put{$\bullet$} at 3.9 2
\put{$\bullet$} at 3.9 -1
\put{$\bullet$} at -.7 -1
\put{$\bullet$} at .6 -1
\put{$\bullet$} at 2.5 -1
\put{$\bullet$} at .6 2
\put{$\bullet$} at 2.5 2
%\put{$\cdots$} at 0 0.5
%\put{$\cdots$} at 3.3 0.5
%\put{$\cdots$} at 0 -5.5
%\put{$\cdots$} at 3.3 -5.5

\plot 2.5 -7 2.5 -4 /
\plot .6 -7 .6 -4 /
\plot -.7 -4 -.7 -7 /
\plot 3.9 -7 3.9 -4 /

\plot 2.5 2 2.5 -1 /
\plot .6 2 .6 -1 /
\plot -.7 -1 -.7 2 /
\plot 3.9 2 3.9 -1 /

\plot 1 -4 1 -1 /
\plot 2 -4 2 -1 /
\plot 1 -1 2 2 /
\plot 1 -1 2 2 /
\plot 1 -4 2 -7 /
\setquadratic

\endpicture}
$$
\begin{proof}
Since all of $T$'s factors have index less than $i$ and all of $T'$'s factors have index greater than $i+1$, we can commute $r_i$ to commute through $T$ and $T'$ to get 

$$T T' p_i = T T' r_i l_i = r_i T  T' l_i.$$ 
%Note that by our diagram decomposition that if we want to move a punctured vertical line over a ``slash" (a $r_i$ or $r_j$), that there must be some horizontal edge between the punctured vertical line and ``slash".  We refer to case 1 (hopping) to move the punctured vertical line over the horizontal edge, and as a result get a punctured slash next to a non-punctured slash, in which we refer to the following case:\\
\end{proof}

%%%%%%%%%%%%%%%%
 
If vertex $i$ is connected to vertex $j$, then $i$ being punctured is equivalent to $j$ being punctured. This fact can be seen from the diagram and is stated formally and proven below.
%Note: The picture was drawn at an earlier time, when the lemma was stated a little differently. To better reflect the current statement of the lemma, the T_a and T_b need to be removed, k should be changed to j, and instead of the diagram on the right having an r_i on top and an l_k on bottom, it should have a p_i on top and a p_j on bottom. --Eliezer, 1/10/12
$$
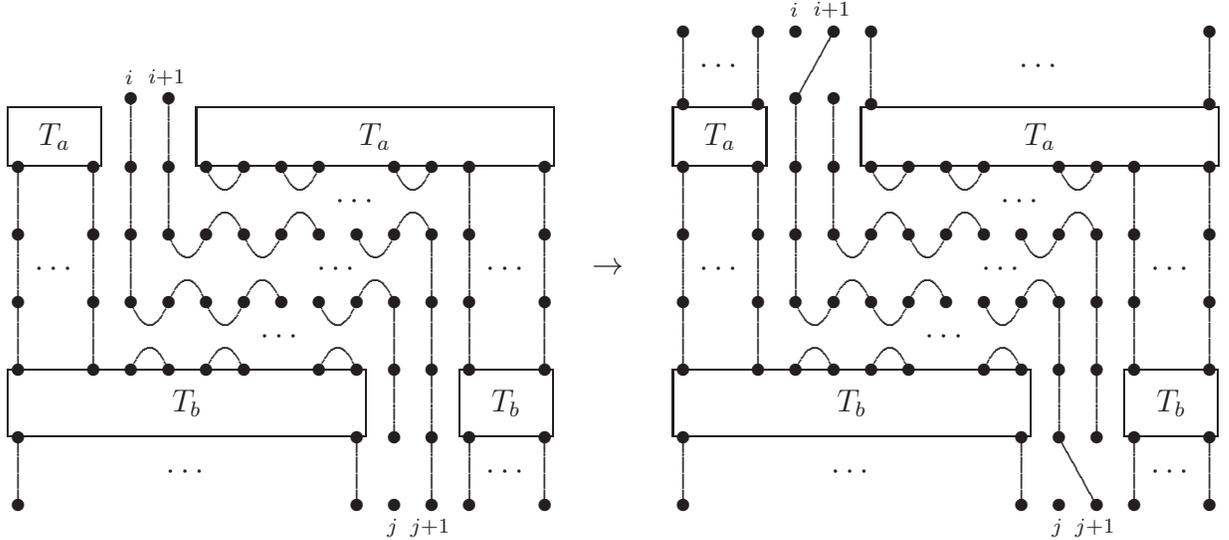
{\beginpicture
\setcoordinatesystem units <0.5cm,0.3cm>
\setplotarea x from 1 to 4, y from -1 to 2

\put{$\begin{picture}(50,50)
\put(0,0){\framebox(35,22){$T_a$}}
\end{picture}$} at 2.5 2
\put{$\begin{picture}(50,50)
\put(0,0){\framebox(135,22){$T_a$}}
\end{picture}$} at 7.5 2
\put{$\begin{picture}(50,50)
\put(0,0){\framebox(135,25){$T_b$}}
\end{picture}$} at 2.5 -10
\put{$\begin{picture}(50,50)
\put(0,0){\framebox(35,25){$T_b$}}
\end{picture}$} at 14.5 -10

\put{$\bullet$} at 1 -16 \put{$\bullet$} at 1 -13 \put{$\bullet$} at 1 -10 \put{$\bullet$} at 1 -7 \put{$\bullet$} at 1 -4 \put{$\bullet$} at  1 -1

\put{$\bullet$} at 3 -10 \put{$\bullet$} at 3 -7 \put{$\bullet$} at 3 -4 \put{$\bullet$} at  3 -1

\put{$\bullet$} at 4 -10 \put{$\bullet$} at 4 -7 \put{$\bullet$} at 4 -4 \put{$\bullet$} at  4 -1 \put{$\bullet$} at 4 2

\put{$\bullet$} at 5 -10 \put{$\bullet$} at 5 -7 \put{$\bullet$} at 5 -4 \put{$\bullet$} at  5 -1 \put{$\bullet$} at 5 2

\put{$\bullet$} at 6 -10 \put{$\bullet$} at 6 -7 \put{$\bullet$} at 6 -4 \put{$\bullet$} at  6 -1

\put{$\bullet$} at 7 -10 \put{$\bullet$} at 7 -7 \put{$\bullet$} at 7 -4 \put{$\bullet$} at  7 -1

\put{$\bullet$} at 8 -7 \put{$\bullet$} at 8 -4 \put{$\bullet$} at  8 -1

\put{$\bullet$} at 9 -10 \put{$\bullet$} at 9 -7 \put{$\bullet$} at 9 -4 \put{$\bullet$} at  9 -1

\put{$\bullet$} at 10 -16 \put{$\bullet$} at 10 -13 \put{$\bullet$} at 10 -10 \put{$\bullet$} at 10 -7 \put{$\bullet$} at 10 -4 %\put{$\bullet$} at  10 -1

\put{$\bullet$} at 11 -16 \put{$\bullet$} at 11 -13 \put{$\bullet$} at 11 -10 \put{$\bullet$} at 11 -7 \put{$\bullet$} at 11 -4 \put{$\bullet$} at  11 -1

\put{$\bullet$} at 12 -16 \put{$\bullet$} at 12 -13 \put{$\bullet$} at 12 -10 \put{$\bullet$} at 12 -7 \put{$\bullet$} at 12 -4 \put{$\bullet$} at  12 -1

\put{$\bullet$} at 13 -16 \put{$\bullet$} at 13 -13 \put{$\bullet$} at 13 -10 \put{$\bullet$} at 13 -7 \put{$\bullet$} at 13 -4 \put{$\bullet$} at  13 -1

\put{$\bullet$} at 15 -16 \put{$\bullet$} at 15 -13 \put{$\bullet$} at 15 -10 \put{$\bullet$} at 15 -7 \put{$\bullet$} at 15 -4 \put{$\bullet$} at  15 -1

\put{$\cdots$} at 2 -5.5
\put{$\cdots$} at 10 -2.5
\put{$\cdots$} at 14 -5.5
\put{$\cdots$} at 14 -14.5
\put{$\cdots$} at 5.5 -14.5
\put{$\cdots$} at 8 -8.5

\put{$\cdots$} at 9.5 -5.5

%\plot  3 2  1 -1 / %(2)
\plot 1 -1 1 -4 /
\plot 1 -4 1 -7 /
\plot 1 -7 1 -10 /
\plot 3 -1 3 -4 /
\plot 3 -4 3 -7 /
\plot 3 -7 3 -10 /
\plot 12 -7 12 -10 /
\plot 12 -10 12 -13 /
\plot 12 -13 12 -16 /
\plot 13 -13 13 -16 /
\plot 15 -13 15 -16 /
\plot 10 -13 10 -16 /
\plot 1 -13 1 -16 /
\plot 11 -7 11 -10 /
\plot 11 -10 11 -13 /

\plot 4 2 4 -1 /
\plot 4 -1 4 -4 /
\plot 4 -4 4 -7 /
\plot 5 2 5 -1 /
\plot 5 -1 5 -4 /
\plot 13 -1 13 -4 /
\plot 15 -1 15 -4 /
\plot  12 -4 12 -7 /
\plot  13 -4 13 -7 /
\plot  15 -4 15 -7 /
\plot 13 -7 13 -10 /
\plot 15 -7 15 -10 /

\setquadratic
\plot 6 -1 6.5 -2 7 -1 /  % (3)
\plot 8 -1 8.5 -2 9 -1 /
\plot 11 -1 11.5 -2 12 -1 /
\plot 5 -4 5.5 -5 6 -4 /
\plot 6 -4 6.5 -3 7 -4 /
\plot 7 -4 7.5 -5 8 -4 /
\plot 8 -4 8.5 -3 9 -4 /
\plot 10 -4 10.5 -5 11 -4 /
\plot 11 -4 11.5 -3 12 -4 /
\plot 4 -7 4.5 -8 5 -7 /
\plot 5 -7 5.5 -6 6 -7 /
\plot 6 -7 6.5 -8 7 -7 /
\plot 7 -7 7.5 -6 8 -7 /
\plot 9 -7 9.5 -8 10 -7 /
\plot 10 -7 10.5 -6 11 -7 /
\plot 4 -10 4.5 -9 5 -10 /
\plot 6 -10 6.5 -9 7 -10 /
\plot 9 -10 9.5 -9 10 -10 /

\put{$\scriptstyle{i\,}$} at 4 3.0 % (4)
\put{$\scriptstyle{i+1\,}$} at 5 3.0
\put{$\scriptstyle{j\,}$} at 11 -17 % (4)
\put{$\scriptstyle{j+1\,}$} at 12 -17

\endpicture}
{\beginpicture
\setcoordinatesystem units <0.5cm,0.3cm>
\setplotarea x from 1 to 4, y from -1 to 5

\put{$\begin{picture}(50,50)
\put(0,0){\framebox(35,22){$T_a$}}
\end{picture}$} at 2.5 2
\put{$\begin{picture}(50,50)
\put(0,0){\framebox(135,22){$T_a$}}
\end{picture}$} at 7.5 2
\put{$\begin{picture}(50,50)
\put(0,0){\framebox(135,25){$T_b$}}
\end{picture}$} at 2.5 -10
\put{$\begin{picture}(50,50)
\put(0,0){\framebox(35,25){$T_b$}}
\end{picture}$} at 14.5 -10

\put{$\rightarrow$} at -1 -5.5
\put{$\bullet$} at 1 -16 \put{$\bullet$} at 1 -13 \put{$\bullet$} at 1 -10 \put{$\bullet$} at 1 -7 \put{$\bullet$} at 1 -4 \put{$\bullet$} at  1 -1 \put{$\bullet$} at  1 5 \put{$\bullet$} at  1 1.75

\put{$\bullet$} at 3 -10 \put{$\bullet$} at 3 -7 \put{$\bullet$} at 3 -4 \put{$\bullet$} at  3 -1 \put{$\bullet$} at  3 1.75 \put{$\bullet$} at  3 5

\put{$\bullet$} at 4 -10 \put{$\bullet$} at 4 -7 \put{$\bullet$} at 4 -4 \put{$\bullet$} at  4 -1 \put{$\bullet$} at 4 2 \put{$\bullet$} at  4 5

\put{$\bullet$} at 5 -10 \put{$\bullet$} at 5 -7 \put{$\bullet$} at 5 -4 \put{$\bullet$} at  5 -1 \put{$\bullet$} at 5 2 \put{$\bullet$} at  5 5

\put{$\bullet$} at 6 -10 \put{$\bullet$} at 6 -7 \put{$\bullet$} at 6 -4 \put{$\bullet$} at  6 -1 \put{$\bullet$} at  6 1.75 \put{$\bullet$} at  6 5

\put{$\bullet$} at 7 -10 \put{$\bullet$} at 7 -7 \put{$\bullet$} at 7 -4 \put{$\bullet$} at  7 -1

\put{$\bullet$} at 8 -7 \put{$\bullet$} at 8 -4 \put{$\bullet$} at  8 -1

\put{$\bullet$} at 9 -10 \put{$\bullet$} at 9 -7 \put{$\bullet$} at 9 -4 \put{$\bullet$} at  9 -1

\put{$\bullet$} at 10 -16 \put{$\bullet$} at 10 -13 \put{$\bullet$} at 10 -10 \put{$\bullet$} at 10 -7 \put{$\bullet$} at 10 -4

\put{$\bullet$} at 11 -16 \put{$\bullet$} at 11 -13 \put{$\bullet$} at 11 -10 \put{$\bullet$} at 11 -7 \put{$\bullet$} at 11 -4 \put{$\bullet$} at  11 -1

\put{$\bullet$} at 12 -16 \put{$\bullet$} at 12 -13 \put{$\bullet$} at 12 -10 \put{$\bullet$} at 12 -7 \put{$\bullet$} at 12 -4 \put{$\bullet$} at  12 -1

\put{$\bullet$} at 13 -16 \put{$\bullet$} at 13 -13 \put{$\bullet$} at 13 -10 \put{$\bullet$} at 13 -7 \put{$\bullet$} at 13 -4 \put{$\bullet$} at  13 -1

\put{$\bullet$} at  15 5 \put{$\bullet$} at 15 -16 \put{$\bullet$} at 15 -13 \put{$\bullet$} at 15 -10 \put{$\bullet$} at 15 -7 \put{$\bullet$} at 15 -4 \put{$\bullet$} at  15 -1 \put{$\bullet$} at  15 1.75

\put{$\cdots$} at 2 -5.5
\put{$\cdots$} at 10 -2.5
\put{$\cdots$} at 14 -5.5
\put{$\cdots$} at 14 -14.5

\put{$\cdots$} at 5.5 -14.5

\put{$\cdots$} at 8 -8.5
\put{$\cdots$} at 2 3.5
\put{$\cdots$} at 10.5 3.5
\put{$\cdots$} at 9.5 -5.5

%\plot  3 2  1 -1 / %(2)
\plot 1 -1 1 -4 /
\plot 1 -4 1 -7 /
\plot 1 -7 1 -10 /
\plot 3 -1 3 -4 /
\plot 3 -4 3 -7 /
\plot 3 -7 3 -10 /
\plot 12 -7 12 -10 /
\plot 12 -10 12 -13 /
\plot 13 -13 13 -16 /
\plot 15 -13 15 -16 /
\plot 10 -13 10 -16 /
\plot 1 -13 1 -16 /
\plot 11 -7 11 -10 /
\plot 11 -10 11 -13 /
\plot 4 2 5 5 /

\plot 4 2 4 -1 /
\plot 4 -1 4 -4 /
\plot 4 -4 4 -7 /
\plot 5 2 5 -1 /
\plot 5 -1 5 -4 /
\plot 13 -1 13 -4 /
\plot 15 -1 15 -4 /
\plot  12 -4 12 -7 /
\plot  13 -4 13 -7 /
\plot  15 -4 15 -7 /
\plot 13 -7 13 -10 /
\plot 15 -7 15 -10 /
\plot 1 2 1 5 /
\plot 15 2 15 5 /
\plot 3 2 3 5 /
\plot 6 2 6 5 /
\plot 11 -13 12 -16 /

\setquadratic
\plot 6 -1 6.5 -2 7 -1 /  % (3)
\plot 8 -1 8.5 -2 9 -1 /
\plot 11 -1 11.5 -2 12 -1 /
\plot 5 -4 5.5 -5 6 -4 /
\plot 6 -4 6.5 -3 7 -4 /
\plot 7 -4 7.5 -5 8 -4 /
\plot 8 -4 8.5 -3 9 -4 /
\plot 10 -4 10.5 -5 11 -4 /
\plot 11 -4 11.5 -3 12 -4 /
\plot 4 -7 4.5 -8 5 -7 /
\plot 5 -7 5.5 -6 6 -7 /
\plot 6 -7 6.5 -8 7 -7 /
\plot 7 -7 7.5 -6 8 -7 /
\plot 9 -7 9.5 -8 10 -7 /
\plot 10 -7 10.5 -6 11 -7 /
\plot 4 -10 4.5 -9 5 -10 /
\plot 6 -10 6.5 -9 7 -10 /
\plot 9 -10 9.5 -9 10 -10 /

\put{$\scriptstyle{i\,}$} at 4 6.0 % (4)
\put{$\scriptstyle{i+1\,}$} at 5 6.0
\put{$\scriptstyle{j\,}$} at 11 -17 % (4)
\put{$\scriptstyle{j+1\,}$} at 12 -17

\endpicture}
$$
\begin{lemma}[Fuse wire]\label{fuse}
%\textcolor{red}{
(i) If $j<i$, $i-j$ is even, and $x=(t_{j+1}t_{j+3}\ldots t_{i-1}) (t_{j}t_{j+2}\ldots t_{i-2}),$ then $xp_i=p_jx=p_ixp_j$.\\
(ii) If $ i<j$, $j-i$ is even, and $x=(t_i t_{i+2}t_{i+4}\ldots t_{j-2}) (t_{i+1}t_{i+3}\ldots t_{j-1}),$ then $xp_i=p_jx=p_ixp_j$.\\
(iii) If $i<j$, $j-i$ is odd, and $x=(t_it_{i+2}\ldots t_{j-1})(t_{i+1}t_{i+3}\ldots t_{j-2})$, then $xp_j=xp_i=xp_jp_i$.\\
(iv) If $j<i$, $j-i$ is odd, and $x=(t_jt_{j+2}\ldots t_{i-1})(t_{j+1}t_{j+3}\ldots t_{i-2})$, then $xp_i=xp_j=xp_ip_j$.
%} end red
\end{lemma}
\begin{proof}
%\textcolor{red}
{
(i) By the commutativity relations, we can rearrange $x$ to be of the form $$x=(t_{j+1}t_{j})(t_{j+3}t_{j+2})(t_{j+5}t_{j+4})\dots (t_{i-3}t_{i-4})(t_{i-1}t_{i-2})$$
This implies that
\begin{eqnarray*}
xp_i&=&(t_{j+1}t_{j})(t_{j+3}t_{j+2})\dots t_{i-5}t_{i-6}t_{i-3}t_{i-4}t_{i-1}t_{i-2}p_i\\
&=&(t_{j+1}t_{j})(t_{j+3}t_{j+2})\dots t_{i-3}t_{i-4}t_{i-1}p_it_{i-2}\textrm{ ($t_jp_i=p_it_j$ for $|j-i|>1$)}\\
&=&(t_{j+1}t_{j})(t_{j+3}t_{j+2})\dots t_{i-3}t_{i-4}t_{i-1}p_{i-1}t_{i-2}\textrm{ ($t_jp_{j+1}=t_jp_j$)}\\
&=&(t_{j+1}t_{j})(t_{j+3}t_{j+2})\dots t_{i-3}t_{i-4}t_{i-1}p_{i-2}t_{i-2}\textrm{ ($e_{j+1}t_j=p_jt_j$)}\\
&=&(t_{j+1}t_{j})(t_{j+3}t_{j+2})\dots t_{i-3}t_{i-4}p_{i-2}t_{i-1}t_{i-2}\textrm{ ($p_jt_{j+1}=t_{j+1}p_j$)}\\
\end{eqnarray*}
}%end red
This process can be repeated until $xp_j$ has been transformed to $p_ix$.
The equation $xp_j=p_ixp_j$ can be derived by first noting that $xp_j=xp_j^2$ and using the above result to move one of the factors of $p_j$ to the far left.\\
(ii) The proof is analogous to that of the previous case.\\

%\textcolor{red}{
(iii) $xp_j=t_i[t_{i+2}t_{i+4}\ldots t_{j-1}t_{i+1}t_{i+3}\ldots t_{j-2}p_j]$. Notice that the factor in brackets is of the form discussed in case (ii). Thus we obtain $xp_j=t_i[p_{i+1} t_{i+2}\ldots t_{j-1}t_{i+1}t_{i+3}t_{j-2}]$
This is in turn equal to $t_i[p_i t_{i+2}\ldots t_{j-1}t_{i+1}t_{i+3}t_{j-2}]$, by the relation $t_i p_i = t_i p_{i+1}$.
We can then commute the $p_i$ to the far right, yielding $xp_j = t_i t_{i+2}\ldots t_{j-1}t_{i+1}t_{i+3}t_{j-2}p_i$ as desired. As discussed at the end of (i), one can easily prove that $x p_i = x p_i p_j$.\\
%}%end red
(iv) The proof is analogous to that of the previous case. 
\end{proof}
%End Eliezer's Section; Begin Megan's Section
\subsubsection{Putting an order on $P_n$}

By putting an order on $P_n$ we can measure how close a given word is to being in $RTL$ form and thereby better describe the actions necessary to perform our decomposition and get the ``dead ends", or non-incident vertices, on the right and incident vertices on the left.\\

%\begin{definition}[ordering on $P_n$]
%Given $P$, and $P'$, two words comprised of generators in $P_n$, we say that the order of $P$ is less than the order of $P'$, $P<P'$, if $|\tau(P)| < |\tau(P')|$.  If $|\tau(P)| = |\tau(P')|$, then $P < P'$ if $$\sum_{x \in \tau(P)} < \sum_{x \in \tau(P')} $$
%\end{definition}

\begin{definition}
Let $<$ be any order on the power set of $\{1,\ldots,n\}$ with the properties:\\
$\bullet\ $If$\ X \subseteq Y$ then $X \leq Y$\\
$\bullet\ $If$\ i,i+1 \notin X$ then $X \cup \{i\} < X \cup \{i+1\}$\\
%\noindent For example we could use graded lexicographic order.
\end{definition}

\begin{definition}
A word $d\in M_n'$ is said to be in \textit{minimal form} if it is in the form $RTL$ for some $R\in RP_n, T\in TL_n$, and $L\in LP_n$ with the property that for all $R'\in RP_n, T'\in TL_n$, and $L'\in LP_n$ such that $d=R'T'L'$
$$\tau(R)\leq \tau(R') \textrm{ and } \beta(L)\leq \beta(L').$$

In this case, $R$ and $L$ are said to be \textit{minimal}.

\end{definition}

\begin{theorem}[ordering on $P$]\label{ordering}
If $P \in P_n$ and $i \in \{1, \dots, n-1\}$ then $\tau (l_iP) \leq \tau (P)$ with equality if and only if $i, i+1 \notin \tau (P)$.
\end{theorem}

\begin{proof}%Yes, I should be proofreading this, because it's in English, but I'm lazy --e
%\textcolor{red}{
Notice that if $i+1 \in \tau(P)$, left multiplication by $r_i$ will send the top of $P$ to the $ith$ index, leaving $\tau(l_i P)$ empty in the index $i$ on top,  thus $\tau(l_i P) < L$, hence $r_i P < P$.
Now if $j+1 \in \tau(P)$, then $r_j P < P$.  Furthermore, if $j \in \tau(P)$, then left multiplication by $r_j$ leaves the $jth$ vertex on top of $r_j P$ empty, hence $r_j P < P$.  Left multiplication by $r_j$ only affects these two top vertices, thus in all other cases $P$ is unaffected. Hence if $j, j+1 \notin P$, then $r_j P = P$, giving us the desired result.
%}% end red
\end{proof}

\begin{theorem}\label{algorithm}
If $d=P_1TP_2$ is minimal, then the following three assertions hold:\\
(1) For all $1 \leq i \leq n$\\
If vertex $i'$ is connected to $j'$ in $T$, then $i'\in\tau(P_2)$ if and only if $j'\in\tau(P_2)$.\\
%If vertex $i'$ is connected to $j$ in $T$, then $i'\in\tau(P_2)$ if and only if $j\in\beta(P_1)$.\\
If vertex $i$ is connected to $j'$ in $T$, then $i\in\beta(P_1)$ if and only if $j'\in\tau(P_2)$.\\
If vertex $i$ is connected to $j$ in $T$, then $i\in\beta(P_1)$ if and only if $j\in\beta(P_1)$.\\
(2) There exists a natural number $k$ such that $\tau(P_2)=\{1,2,\dots,k\}$\\
(3) There exists a natural number $m$ such that $\beta(P_1)=\{1,2,\dots,m\}$\\
\end{theorem}
\begin{proof}
(1) If any of the implications are false, then the hypothesis for at least one of the cases of lemma \ref{fuse} are satisfied, which implies that $d$ is not in minimal form.

%\textcolor{red}{\textbf{Begin Megan--don't expect to reach "End Megan" for a while}}

(2) If $\tau (P_2) \neq \{1, \ldots, k \}$, for any k, there exists an $i$ such that $i \notin \tau (P_2)$ and $i+1 \in \tau(P_2).$ There are five ways in which this can happen.  In each case we can prove that $P_1TP_2=P_1'T'P_2'$, where $P_1'\leq P_1$ and $P_2'<P_2$, implying that $P_1TP_2$ was not minimal.

Case 1. $i'$ is connected to $j'$ with $i<j$.\\
Then $T= T _1(t_{j-1} t_{j-3}\cdots t_i) (t_{j-2} t_{j-4} \cdots t_{i+1}) T_2$ where $T_1 \in \langle t_1 , \ldots t_{n-1} \rangle, T_2\in \langle t_{i+1} \ldots t_{j-2} \rangle$.
Consider $Te_i$.  %For all $t_k \in T_2$, we have that $|k-i| \geq 2$.
We commute every Temperley-Lieb generator in $T_2$ with $e_i$.  With the $e_i$ in place, notice that the hypothesis for Theorem \ref{hopping} are satisfied where $k$ in Theorem \ref{hopping} is equal to $j-1$. We hop the dead end at $i$ to $j-1$.  This yields:
\begin{eqnarray*}
Te_i &=& T _1(t_{j-1} t_{j-3}\cdots t_i) (t_{j-2} t_{j-4} \cdots t_{i+1}) T_2e_i\\
&=& T _1(t_{j-1} t_{j-3}\cdots t_i) (t_{j-2} t_{j-4} \cdots t_{i+1})e_i T_2\\
&=& T_1(t_{j-1} t_{j-3}\cdots t_i) (t_{j-2} t_{j-4} \ldots t_{i+1}) (t_{j-3} t_{j-5} \ldots t_{i})(l_{j-2}l_{j-3} \ldots l_i) T_2.
\end{eqnarray*}

Subsequently, we slide (lemma \ref{slide}) to get:
$$Te_i=T_1(t_{j-1} t_{j-3}\cdots t_i) (t_{j-2} t_{j-4} \ldots t_{i+1}) (t_{j-3} t_{j-5} \ldots t_{i})T_{2}'(l_{j-2}l_{j-3} \ldots l_i)$$
where $T_{2}' \in \langle t_{i} \ldots t_{j-3} \rangle $.

In conclusion, $$d=P_1TP_2=P_1T(l_{j-2}l_{j-3} \ldots l_i)P_2$$ and by Theorem \ref{ordering} $\tau((l_{j-2}l_{j-3} \ldots l_i)P_2 )<\tau(P_2)$ since $i+1\notin P_2$.

%Now we can simplify the product of $t$'s by commuting and using the relation $t_i t_j t_i = t_i$ if $|i-j| =1$.  This allows us to commute $t_{j-1}$ over to $r_{j-2}$.  We then use the relation $r_i = e_{i+1} l_i$ to produce an $e_{j-1}$.  Now we can burrow (theorem \ref{burrow}) as follows:
%\begin{eqnarray*}
%&& T_1(t_{j-1} t_{j-3}\cdots t_i)  (t_{j-2} t_{j-4} \ldots t_{i+1}) (t_{j-3} t_{j-5} \ldots t_{i})T_{2}'(l_{j-2}l_{j-3} \ldots l_i)\\
%&=& T_1 t_{j-1} (t_{j-3} t_{j-2} t_{j-3} \cdots t_{i} t_{i+1} t_i)T_{2}'(l_{j-2}l_{j-3} \ldots l_i)\\
%&=& T_1(t_{j-1} t_{j-3}\cdots t_i)  T_{2}'(l_{j-2}l_{j-3} \ldots l_i)\\
%&=& T_1(t_{j-3} t_{j-5} \cdots t_i)  T_{2}' t_{j-1}l_{j-2}l_{j-3} \ldots l_i\\
%&=& T_1(t_{j-3}t_{j-5} \cdots t_i)  T_{2}'t_{j-1}e_{j-1} l_{j-2}l_{j-3} \ldots l_i\\
%&=& T_1(t_{j-3}t_{j-5} \cdots t_i)  T_{2}'t_{j-1}t_j l_{j-1} l_j l_{j-2}l_{j-3} \ldots l_i\\
%&=&T_1(t_{j-3}t_{j-5} \cdots t_i)  T_{2}'t_{j-1}t_j (l_{j-1} l_j l_{j-2}l_{j-3} \ldots l_i).
%\end{eqnarray*}
Case 2. $i'$ is connected to $j'$ with $i>j$.\\
In this case we have $$T= T_1 (t_{i-1} t_{i-3}\cdots t_j) (t_{i-2} t_{i-4} \cdots t_{j+1}) T_2$$ where $T_1 \in \langle t_1 , \ldots t_{n-1} \rangle$ and  $T_2\in \langle t_{j+1} \ldots t_{i-2} \rangle$. Since $i\notin\tau(P_2)$ we have $d=P_1Tp_iP_2=P_1Tp_ip_jP_2$ (by Lemma \ref{fuse}). We can now obtain $P_2'$ such that $\tau(P_2')<\tau(P_2)$ by referring to the previous case.

Case 3.  $i'$ is connected to $j$ with $j<i$\\
In the case, $T$ is of the form $$T=T _1(t_{j+1} t_{j+3}\cdots t_{i-1}) (t_j t_{j+2} \cdots t_{i-2}) T_2$$ where $T_1 \in \langle t_{j+1}, \dots t_{n-1} \rangle$ and $T_2 \in \langle t_1 \ldots t_{i-2} \rangle$.
Consider $Tp_i$.  As in case 1 we can commute $p_i$ through $T_2$.  Notice that $j<i, i-j$ is even, and we have a divisor $x=(t_{j+1} t_{j+3} \cdots t_{i-1})(t_j t_{j+2} \cdots t_{i-2})$.  Thus by lemma \ref{fuse}, we get:
\begin{eqnarray*}
Tp_i &=& T _1(t_{j+1} t_{j+3}\cdots t_{i-1}) (t_j t_{j+2} \cdots t_{i-2}) T_2 p_i \\
&=& T _1(t_{j+1} t_{j+3}\cdots t_{i-1}) (t_j t_{j+2} \cdots t_{i-2}) p_i T_2\\
&=& T_1 p_j (t_{j+1} t_{j+3}\cdots t_{i-1}) (t_j t_{j+2} \cdots t_{i-2}) T_2
\end{eqnarray*}

Observe that taking the antiisomorphism of Theorem \ref{hopping} (hop) we obtain:
\begin{eqnarray*}
[(t_{k-1} t_{k-3} \ldots t_{i+1}) p_i]^* &=& [(t_{k-1} t_{k-3} \ldots t_{i+1}) (t_{k-2} t_{k-4} \ldots t_{i})(l_{k-1}l_{k-2} \ldots l_i)]^*\\
p_i (t_{i+1} t_{i+3} \ldots t_{k-1}) &=&
(r_{i}r_{i+1} \ldots r_{k-1})(t_{i} t_{i+2} \ldots t_{k-2})(t_{i+1} t_{i+3} \ldots t_{k-1})
\end{eqnarray*}

Hence by applying Theorem \ref{hopping},
\begin{eqnarray*}
Tp_i &=& T_1p_j(t_{j+1} t_{j+3}\cdots t_{i-1})(t_j t_{j+2} \cdots t_{i-2}) T_2\\
&=& T_1 (r_{j}r_{j+1} \ldots r_{i-1})(t_{j} t_{j+2} \ldots t_{i-2})(t_{j+1} t_{j+3} \ldots t_{i-1})(t_j t_{j+2} \cdots t_{i-2}) T_2.
\end{eqnarray*}

As in case 1 we can use the relation $t_a t_k t_a = t_a$ if $|a-k| = 1$ to simplify the product of $t$'s.  Now we use the relation $r_k = r_k p_{k+1}$ and apply Lemma \ref{wallslide1} to produce all the $r$'s between $j$ and $n-1$.  So,
\begin{eqnarray*}
&=& T_1(r_{j}r_{j+1} \ldots r_{i-1})(t_j t_{j+2} \cdots t_{i-2})T_2\\
&=& T_1 (r_{j}r_{j+1} \ldots r_{i-1}) p_i (t_j t_{j+2} \cdots t_{i-2})T_2\\
&=& T_1 (r_{j}r_{j+1} \ldots r_{i-1}) r_i l_i (t_j t_{j+2} \cdots t_{i-2})T_2\\
&=& T_1 (r_{j}r_{j+1} \ldots r_{i-1})(r_i \dots r_{n-1})(l_{n-1}\ldots l_i)(t_j t_{j+2} \cdots t_{i-2})T_2\\
&=& T_1 (r_{j}r_{j+1} \ldots r_{i-1})(r_i \cdots r_{n-1})(t_j t_{j+2} \cdots t_{i-2}) T_2 (l_{n-1} \cdots l_{i}).
\end{eqnarray*}

Now the hypothesis for slide (Theorem \ref{slide}) are satisfied, so we obtain:
$$T p_i = (r_i \cdots r_{n-1}) T_1' (t_j t_{j+2} \cdots t_{i-2} T_2 (l_{n-1} \cdots l_{i})$$
where $T_1' \in \langle t_2 \ldots t_{i-1} \rangle$.
In conclusion, $$d=P_1TP_2=P_1Tp_iP_2=P_1(r_i \cdots r_{n-1}) T_1' (t_j t_{j+2} \cdots t_{i-2} T_2 (l_{n-1} \cdots l_{i})P_2.$$ Notice that $\beta(P_1r_i \cdots r_{n-1})\leq \beta(P_1)$ and $\tau(l_{n-1} \cdots l_{i}P_2)<\tau(P_2)$ by Theorem \ref{ordering}, since $i+1\notin\tau(P_2)$.

Case 4. $i'$ is connected to $j$ with $j>i$\\
In this case $T$ can be written in the form $T_1 (t_{j-2} t_{j-4} \cdots t_i)(t_{j-1} t_{j-3} \cdots t_{i+1}) T_2$ where $T_1 \in \langle t_1, \ldots t_{j-2} \rangle, T_2 \in \langle t_{i+1} \ldots t_{n-1} \rangle$.
Consider $Tp_i$.  Every element in $T_2$ commutes with $p_i$.  Notice that the hypothesis in Theorem \ref{hopping} are satisfied for $k=j$.  Thus,
\begin{eqnarray*}
Tp_i &=& T_1 (t_{j-2} t_{j-4} \cdots t_i)(t_{j-1} t_{j-3} \cdots t_{i+1}) T_2 p_i\\
&=& T_1 (t_{j-2} t_{j-4} \cdots t_i)(t_{j-1} t_{j-3} \cdots t_{i+1}) p_i T_2\\
&=& T_1(t_{j-2} t_{j-4} \cdots t_i)(t_{j-1} t_{j-3} \cdots t_{i+1})( t_{j-2} t_{j-4} \cdots t_i) (l_{j-1} l_{j-2} \cdots l_i) T_2.
\end{eqnarray*}

Using the relations $t_it_jt_i=t_i$ if $|i-j|=1$ and $t_it_j=t_jt_i$ if $|i-j| \geq 2$ as in case 1, we can simplify the product of $t$'s to $(t_it_{i+2}\cdots t_{j-2})$.  Now we use the relation $r_i = p_{i+1} l_i$ and apply wallslide (Lemma \ref{wallslide1}) to produce all the $l$'s between $j$ and $n-1$, so that we will be able to slide the product of $l$'s (Theorem \ref{slide}) through $T_2$.  So we have:
\begin{eqnarray*}
&=&T_1(t_{j-2} t_{j-4} \cdots t_i)(l_{j-1}l_{j-2}\cdots l_i) T_2\\
&=&T_1(t_{j-2} t_{j-4} \cdots t_i)p_j(l_{j-1}l_{j-2}\cdots l_i) T_2\\
&=&r_jT_1(t_{j-2} t_{j-4} \cdots t_i)(l_jl_{j-1}l_{j-2}\cdots l_i) T_2\\
&=&r_jT_1(t_{j-2} t_{j-4} \cdots t_i)p_{j+1}(l_jl_{j-1}l_{j-2}\cdots l_i) T_2\\
&=&r_jr_{j+1}T_1(t_{j-2} t_{j-4} \cdots t_i)(l_{j+1}l_jl_{j-1}l_{j-2}\cdots l_i) T_2.\\
%&&\\ %Should we use an induction argument instead??
%&=&(r_jr_{j+1}\cdots r_{n-2})T_1(t_{j-2} t_{j-4} \cdots t_i)(l_{n-2}\cdots l_jl_{j-1}l_{j-2}\cdots l_i) T_2\\
%&=&(r_jr_{j+1}\cdots r_{n-2})T_1(t_{j-2} t_{j-4} \cdots t_i)p_{n-1}(l_{n-2}\cdots l_jl_{j-1}l_{j-2}\cdots l_i) T_2\\
%&=&(r_jr_{j+1}\cdots r_{n-2}r_{n-1})T_1(t_{j-2} t_{j-4} \cdots t_i)(l_{n-1}l_{n-2}\cdots l_jl_{j-1}l_{j-2}\cdots l_i) T_2
\end{eqnarray*}
We can continue this process until we arrive at:\\
$(r_jr_{j+1}\cdots r_{n-2}r_{n-1})T_1(t_{j-2} t_{j-4} \cdots t_i)(l_{n-1}l_{n-2}\cdots l_jl_{j-1}l_{j-2}\cdots l_i) T_2$.\\

Now the hypotheses for slide (Lemma \ref{slide}) are satisfied, we obtain:
$$Tp_i=(r_{n-1}r_{n-2}\cdots r_j)T_1(t_it_{i+2}\cdots t_{j-2})T_2' (l_{n-1}l_{n-2}\cdots l_i)$$ where $T_2' \in \langle t_i \ldots t_{n-2} \rangle$.

In conclusion, $$d=P_1Tp_iP_2=P_1(r_{n-1}r_{n-2}\cdots r_j)T_1(t_it_{i+2}\cdots t_{j-2})T_2' (l_{n-1}l_{n-2}\cdots l_i)P_2.$$ Notice that $\tau(P_1r_{n-1}r_{n-2}\cdots r_j)\leq\tau(P_1)$ and that $\tau((l_{n-1}l_{n-2}\cdots l_i)P_2)<\tau(P_2)$, implying that the from $P_1TP_2$ was not minimal.

Case 5: $i'$ is connected to $j$ on top with $j=i$\\
$T$ can be written as $T= T_1 T_2$ where  $T_1 \in \langle t_{1}, t_{2}, \ldots, t_{i-2} \rangle$,  $T_2 \in \langle t_{i+1}, t_{i+2}, \ldots, t_{n-1} \rangle$\\

The proof for this case is nearly identical to that of Case 4.  We apply the hop relation to $p_i$ to get it through $t_{i+1}$ if it appears in the product (otherwise we perform a wallslide), then produce $r_i$ and $l_i$ and commute $r_i$ through $T_1$, and furthermore, obtain some product of $l$'s and apply Theorem \ref{slide} to slide them through $T_2$.

\textit{Case 2}: If $T_2$ does not contain any factors of $t_{i+1}$, we apply wallslide (Lemma \ref{wallslide1}) to get $T p_i = r_i T_2 l_i = r_i T l_i$.  Note that we can commute the $r_i$ through $T_2$ since $T_2 \in \langle t_{i+2}, \ldots, t_{n-1} \rangle$, and $T_1 \in \langle t_1, \ldots t_{i-2} \rangle$
Thus $d=P_1TP_2=P_1r_iTl_iP_2$ and $\tau(P_1r_i)\leq\tau(P_1)$ and $\tau(l_iP_1)<\tau(P_1)$, by theorem \ref{ordering} since $i+1\in\tau(L)$ 

%\textcolor{red}{\textbf{End Megan}}

If (3) does not hold, we can prove that $d$ is not minimal using analogous reasoning.
\end{proof}

\subsubsection{Standard Form}

In this section we prove theorem \ref{std}, namely,  given a word $d\in M_n'$ written in the form $RTL$, where $R$ and $L$ are minimal, $d$ can be written as a word in standard form.  Note that the difference between $RTL$ and standard form concerns edges in $T$ whose endpoints are empty vertices of $R$ or $L$.  The following lemma covers an extreme case.
%If a Temperley-Lieb diagram is multiplied on the right and left by planar rook diagrams whose vertices are are all empty past a certain vertex $k$, and no strand with index greater than or equal to $k$ is connected to any other vertex with index less than $k$, then it does not matter what is happening in the Temperley-Lieb diagram after vertex $k$. This fact about diagrams motivates the following lemma regarding words in $M_n(x)$.

\begin{lemma}\label{t death}
 Let $T$ be an arbitrary element of $TL_n$ composed of generators whose indices are greater than or equal to some natural number $k$, with $k< n$. Let $E = \prod_{i=k}^{n}p_i$. Then $$E\cdot T \cdot E=E.$$
\end{lemma}

\begin{proof} We proceed by induction on the number of letters in $T$. First consider the base case, when $T=t_i$ for some $i$ between $k$ and $n-1$. Then we obtain  $ETE=Ep_it_ip_iE=E$ by the relations $p_i^2=p_i, p_ip_j=p_jp_i$ and $p_i t_i p_i= p_ip_{i+1}$  (from relation 9).\\

Assume that $ETE=E$ whenever $T$ is the product of $m$ letters. We wish to show that $ETt_iE=ETE$ for any $i$ greater than or equal to $k$, thereby proving that $ETE=E$ whenever $T$ is the product of $m+1$ letters. Consider the following cases for what the $i'$th vertex of $T$ is connected to to.

%\textcolor{red}{\textbf{Begin Megan}}
Case 1: $i'$ is connected to $j'$ with $i'<j'-1$. Then $T$ is of the form $$T=T_1 (t_i t_{i+2}\cdots t_{j-1}) (t_{i+1} t_{i+3} \cdots t_{j-2}) T_2$$ where $T_1 \in \langle t_1 , \ldots t_{n-1} \rangle, T_2\in \langle t_{i+1} \ldots t_{j-2} \rangle$.
\begin{eqnarray*}
ETt_iE&=&ET_1 (t_i t_{i+2}\cdots t_{j-1}) (t_{i+1} t_{i+3} \cdots t_{j-2}) T_2t_iE\\
&=&ET_1 (t_i t_{i+2}\cdots t_{j-1}) (t_{i+1} t_{i+3} \cdots t_{j-2})p_j T_2t_iE\ \textrm{(by $p_i^2=p_i, p_ip_j=p_jp_i$)} \\
&=&ET_1 (t_i t_{i+2}\cdots t_{j-1}) (t_{i+1} t_{i+3} \cdots t_{j-2})p_jp_i T_2t_iE\ \textrm{(by theorem \ref{fuse}\ part (iii))}\\
&=&ET_1 (t_i t_{i+2}\cdots t_{j-1}) (t_{i+1} t_{i+3} \cdots t_{j-2})p_j T_2p_it_ip_iE\ \textrm{($p_i$ commutes with any $t$ in $T_2$)}\\
&=&ET_1 (t_i t_{i+2}\cdots t_{j-1}) (t_{i+1} t_{i+3} \cdots t_{j-2})p_j T_2E\  \textrm{(since $p_i t_i p_i = p_i$)} \\
&=&ETE=E\textrm{ (by the inductive hypothesis).}
\end {eqnarray*}

Case 2: $i'$ is connected to $j'$ with $k<j'<i'$.\\ Then $T$ is of the form $T_1 (t_j t_{j+2}\cdots t_{i-1}) (t_{j+1} t_{j+3} \cdots t_{i-2}) T_2$, where $T_1 \in \langle t_1 , \ldots t_{n-1} \rangle, T_2\in \langle t_{j+1} \ldots t_{i-2} \rangle$.
\begin{eqnarray*}
ETt_iE&=&ET_1 (t_j t_{j+2}\cdots t_{i-1}) (t_{j+1} t_{j+3} \cdots t_{i-2}) T_2t_iE\\
&=&ET_1 (t_j t_{j+2}\cdots t_{i-1}) (t_{j+1} t_{j+3} \cdots t_{i-2}) T_2t_ip_jE\ \textrm{(by $p_j^2=p_j, p_ip_j=p_jp_i$)}\\
&=&ET_1 (t_j t_{j+2}\cdots t_{i-1}) (t_{j+1} t_{j+3} \cdots t_{i-2})p_j T_2t_iE\ \textrm{(by commuting)}\\
&=&ET_1 (t_j t_{j+2}\cdots t_{i-1}) (t_{j+1} t_{j+3} \cdots t_{i-2})p_i T_2t_iE \textrm{ (by Theorem \ref{fuse}\ part\ (iv))}\\
&=&ET_1 (t_j t_{j+2}\cdots t_{i-1}) (t_{j+1} t_{j+3} \cdots t_{i-2}) T_2p_it_i p_iE\ \textrm{(by $p_i^2=p_i$, commuting)}\\
&=&ET_1 (t_j t_{j+2}\cdots t_{i-1}) (t_{j+1} t_{j+3} \cdots t_{i-2}) T_2E\ \textrm{(since $p_i t_i p_i = p_i$)}\\
&=&ETE=E.
\end{eqnarray*}

Case 3: $i'$ is connected to $j$, with $i'<j$.\\ Then $T$ is of the form $T_1 (t_i t_{i+2} \cdots t_{j-2})(t_{i+1} t_{i+3} \cdots t_{j-1}) T_2$ where $T_1 \in \langle t_1, \ldots t_{j-2} \rangle, T_2 \in \langle t_{i+1} \ldots t_{n-1} \rangle $
\begin{eqnarray*}
ETt_iE&=&ET_1 (t_i t_{i+2} \cdots t_{j-2})(t_{i+1} t_{i+3} \cdots t_{j-1}) T_2 t_i E\\
&=&ET_1 p_j (t_i t_{i+2} \cdots t_{j-2})(t_{i+1} t_{i+3} \cdots t_{j-1}) T_2 t_i E\ \textrm{(by $p_j^2=p_j$, commuting)}\\
&=& ET_1 (t_i t_{i+2} \cdots t_{j-2})(t_{i+1} t_{i+3} \cdots t_{j-1}) p_i T_2 t_i E \textrm{ (by Theorem \ref{fuse}\ part\ (i))}\\
&=& ET_1 (t_i t_{i+2} \cdots t_{j-2})(t_{i+1} t_{i+3} \cdots t_{j-1}) T_2 p_i t_i p_i E\ \textrm{(by $p_i^2=p_i$, commuting)}\\
&=& ET_1 (t_i t_{i+2} \cdots t_{j-2})(t_{i+1} t_{i+3} \cdots t_{j-1}) T_2 E\ \textrm{(since $p_i t_i p_i = p_i$)}\\
&=& ETE=E.
\end{eqnarray*}
%\textcolor{red}{\textbf{End Megan}}

Case 4: $i'$ is connected to $j$, with $j<i$. The proof of this case is analogous to Case 3.

Case 5. $i'$ is connected to $i$. This implies that $T$ does not contain $t_i$ or $t_{i-1}$. Thus we obtain
$ETt_iE=ETp_it_ip_iE=ETE=E$, since $p_i t_{i+1} = t_{i+1} p_i$.

Case 6: $i'$ is connected to $(i+1)'$. This implies by section \ref{okay} that $T=T't_i$, for some $T'$. Thus we obtain $$ETt_iE=ET't_i^2E=ET't_iE=ETE=E$$

\end{proof}

%If a Temperley-Lieb diagram is multiplied on bottom by a planar rook diagram whose vertices are all empty past a certain vertex k, and all vertices on bottom with index greater than or equal to k are connected only to each other, then it does not matter what the bottom of the Temperley-Lieb diagram looks like past k. This fact about diagrams motivates the following lemma.
%
%\begin{lemma} $(t_i t_{i+2} \cdots t_k)(p_{k+2}p_{k+3} \cdots p_n)T(p_i p_{i+1} \cdots p_n) = (t_it_{i+2} \cdots t_k)(p_ip_{i+1} \cdots p_n)$, where $T \in \langle t_i, \ldots, t_{n-1} \rangle$.
%\end{lemma}
%\begin{proof} Recall the relation $t_j p_j t_j = t_j p_j p_{j+1} t_j = t_j$.  We can apply this relation to get: \\
%\begin{eqnarray*}
%&&(t_i t_{i+2} \cdots t_k)(p_{k+2}p_{k+3} \cdots p_n)T(p_ip_{i+1} \cdots p_n)\\
%&=&(t_i t_{i+2}\cdots t_k)(p_ip_{i+1}\cdots p_{k+1})(t_i t_{i+2} \cdots t_k)(p_{k+2}p_{k+3} \cdots p_n)T(p_ip_{i+1}\cdots p_n)\\
%&=&(t_i t_{i+2} \cdots t_k)(p_ip_{i+1}\cdots p_n)(t_i t_{i+2} \cdots t_k)T(p_ip_{i+1} \cdots p_n)\\
%&=&(t_i t_{i+2} \cdots t_k)(p_ip_{i+1} \cdots p_n)\textrm { (by \ref{t death})}
%\end{eqnarray*}
%\end{proof}

\begin{theorem}\label{std} Let $d$ be a word in $M_n'$ in the form $RTL$ where $R$ and $L$ are minimal. Then $d$ can be put into standard form.
\end{theorem}
\begin{proof}
We consider two cases.

Case 1: $|\beta(R)|\neq|\tau(L)|$
 Assume without loss of generality that $|\tau(L)|<|\beta(R)|$.
This means that no vertex on bottom with index greater than $|\tau(L)|$ can be connected to a vertex on bottom whose index is less than or equal to $|\tau(L)|$ or any vertex on top with index less than or equal to $|\beta(R)|$. Similarly, no vertex on top with index greater than $|\beta(R)|$ can be connected to a vertex on top whose index is less than or equal to $|\beta(R)|$ or any vertex on bottom with index less than or equal to $|\tau(L)|$. (If this were not the case, $L$ or $R$ would not be minimal, as proven in theorem \ref{algorithm}.)
%\textcolor{red}{By appealing to Temperley-Lieb diagrammatic intuition, this means that $T$ can be written in the form $T_a t_{|\tau(L)|+1}t_{|\tau(L)|+3}\cdots t_{|\beta(R)|-1}T_b$, where $T_a\in \langle t_1, t_2, \ldots, t_{|\beta(R)|-1}\rangle$ and $T_b\in\langle t_{|\tau(L)|},t_{|\tau(L)|+1},\ldots, t_{n-1}\rangle$.
%My diagrammatic intuition has not fully convinced me that this is true, but I trust Stephan.
%}%end red
%\textcolor{red}{
The word $d$ can be written in the form
$$d=R'p_{|\beta(R)|+1}p_{|\beta(R)|+2}\cdots p_n T_a t_{|\tau(L)|+1}t_{|\tau(L)|+3}\cdots t_{|\beta(R)|-1}T_b p_{|\tau(L)|+1}p_{|\tau(L)|+2}\cdots p_n L'$$
for some $R'\in RP_n$ and some $L'\in LP_n$.
\begin{eqnarray*}
d&=&R'T_a t_{|\tau(L)|+1}t_{|\tau(L)|+3}\cdots t_{|\beta(R)|-1} p_{|\beta(R)|+1}p_{|\beta(R)|+2}\cdots p_nT_b p_{|\tau(L)|+1}p_{|\tau(L)|+2}\cdots p_nL'\\
&=&R'T_a t_{|\tau(L)|+1}t_{|\tau(L)|+3}\cdots t_{|\beta(R)|-1}p_{|\tau(L)|+1}p_{|\tau(L)|+2}\cdots p_nL' \textrm{ (by lemma \ref{t death})}
\end{eqnarray*}
which is in standard form.%}%end red

Case 2:  $|\beta(R)|=|\tau(L)|=k$

Using similar reasoning as above, no vertex whose index is greater than $k$ can be connected to any vertex with index less than or equal to $k$. This implies that $T$ can be written without any factors of $t_{k}$. By the commutativity relations, $d$ can be written in the from $RT_a p_{k+1}p_{k+2}\cdots p_n T_b p_{k+1} p_{k+2}\cdots p_nL$, where $T_a\in \langle t_1, t_2,\ldots t_{k-1}\rangle$, and $T_b\in \langle t_{k+1},t_{k+2}\ldots, t_{n-1}\rangle$, which is equal to $RT_a p_{k+1}p_{k+2}\cdots p_nL$ by lemma \ref{t death}, which in standard form.
\end{proof}

\begin{theorem}\label{induction}
Every word in $M_n'$ is equivalent to a standard word in $M_n'$.
\end{theorem}
\begin{proof}
We proceed by induction. Obviously  every single-letter word is equivalent to a standard word. Assume that every word of $k$ letters is equivalent to a word in standard form. Let $d$ be an arbitrary word of $k+1$ letters. Let $d_0$ be the product of the first $k$ letters of $d$, and let $x$ be the final letter of $d$. By the inductive hypothesis, $d_0$ is equivalent to a standard word. If $x\in P_n$, then $d_0x$ is already in PTP form. If $x\in TL_n$, then, by Lemma \ref{rtlx to ptp}, $d_0x$ can be put into $PTP$ form. %I understand that I should cite some specific theorem rather than "`section 3.4.2"', but there is not specific theorem to cite.
Thus, $d$ is equivalent to a word in $PTP$ from. By theorems \ref{algorithm} and \ref{std}, this implies that $d$ is equivalent to a word in standard from. Thus every string of $k+1$ letters is equivalent to a word in standard form.

Thus every word in $M_n'$ is equivalent to a word in standard from.
\end{proof}
\begin{theorem} The monoid ${M_n'}$ is isomorphic to $M_n$.
\end{theorem}
\begin{proof}
By theorem \ref{induction}, we know that the number of distinct words in ${M_n'}$ is equal to the number of standard words. Clearly there is a one-to-one correspondence between words in standard form and diagrams in $M_n$ which are in standard form. Every diagram in $M_n$ is equivalent to a diagram in standard from. Thus $|M_n'|=|{M_n}|$. This implies that the homomorphism $\phi$ is an isomorphism.
\end{proof}

\section{Conclusion}
In this paper we have introduced the rook, planar rook, left and right planar rook monoids, the Temperley-Lieb algebra, and the Motzkin Monoid.  We have given a complete presentation of the Motzkin Monoid through a set of moves (hop, burrow, wallslide), whose proof at times can get very involved.   Monoid presentations are very important in order for us to understand a monoid, and in order for us to define morphisms to other monoids.  The Motzkin Monoid is used extensively in \cite{motzkin}.  Another example of a presentation is completed in \cite{penny}.  In this paper, David Penneys gives a presentation of the Annular Temperley-Lieb category.

\section{Acknowledgments}

The mathematical content described in this paper is derived from a collaboration
over the summer of 2011 at the REU program at the University of California, Santa Barbara.  The group consisted of Megan Ly (Loyola Marymount University), Eliezer Posner (City College of New York), and Kristofer Hatch (University of California, Santa Barbara).\\

	The REU was supported by the National Science Foundation Research Experience for Undergraduates Program (NSF-REU Grant DMS-0852065), and was conducted under the direction of our advisor, Stephen Bigelow.

% Now we type up the bibliography. Do not edit the next line.

% End of document.
\end{document}